\documentclass[11pt, a4paper]{article}

\usepackage{pb-diagram, lamsarrow,pb-lams}
\usepackage{amssymb}
\usepackage{amsthm}
\usepackage{amsmath}
\usepackage[T1]{fontenc}
\usepackage{latexsym}
\usepackage{textcomp}
\usepackage[ansinew]{inputenc}
\usepackage{exscale}
\usepackage{graphicx}
\usepackage{enumerate}

    \newcommand{\Tr}{\mathrm{Tr\,}}

    \newcommand{\aug}{\mathrm{aug\,}}

    \newcommand{\irr}{\mathrm{Irr\,}}

    \newcommand{\Ann}{\mathrm{Ann}}
    \newcommand{\Fitt}{\mathrm{Fitt}}

    \newcommand{\nr}{\mathrm{nr}}
    \newcommand{\Gl}{\mathrm{Gl}}
    \newcommand{\rad}{\mathrm{rad}}

    \newcommand{\norm}{\mathrm{norm}}
    \newcommand{\gldim}{\mathrm{gl.dim~}}

    \newcommand{\example}{\sn{\sc Example.} \ }

    \newcommand{\qpc} {\Q_p^{\mr{c}}}

    \newcommand{\arat} {\ar@}
    \newcommand{\inj} {\arat{^{(}->}}
    \newcommand{\sur} {\arat{->>}}
    \newcommand{\equal} {\arat{=}}
    \def\barat[#1] {\arat{->}'[#1][#1#1]}
    \def\binj[#1] {\inj'[#1][#1#1]}
    \def\bsur[#1] {\sur'[#1][#1#1]}
    \def\bequal[#1] {\equal'[#1][#1#1]}

    \newcommand{\beq}{\begin{equation}}
    \newcommand{\eeq}{\end{equation}}

    \newcommand{\mc}{\mathcal}

     \newcommand{\fa}{\goth{a}}

     \newcommand{\fo}{\goth{o}}
     \newcommand{\fp}{\goth{p}}

     \newcommand{\fD}{\goth{D}}

     \newcommand{\fM}{\goth{M}}

     \newcommand{\fP}{\goth{P}}
     \newcommand{\fQ}{\goth{Q}}

     \newcommand{\F}{\mathbb{F}}
     
     \newcommand{\N}{\mathbb{N}}
     \newcommand{\Q}{\mathbb{Q}}
     
     \newcommand{\Z}{\mathbb{Z}}

    \newcommand{\ti}[1]{\tilde{#1}}

    \newcommand{\me}{^{-1}}
    \newcommand{\mal}{^{\times}}

    \newcommand{\mr}{\mathrm}

    \newcommand{\zp}{{\mathbb{Z}_p}}
    
    \newcommand{\qp}{{\mathbb{Q}_p}}

    \newcommand{\into}{\rightarrowtail}
    \newcommand{\onto}{\twoheadrightarrow}
    \newcommand{\lto}{\longrightarrow}

    \newcommand{\ga}{\gamma}
    \newcommand{\Ga}{\Gamma}
    \newcommand{\si}{\sigma}
    \newcommand{\Si}{\Sigma}
    \newcommand{\de}{\delta}
    \newcommand{\De}{\Delta}
    
    \newcommand{\la}{\lambda}
    \newcommand{\La}{\Lambda}

    \newcommand{\al}{\alpha}
    \newcommand{\ve}{\varepsilon}

    \newcommand{\sn}{\par\smallskip\noindent}

    \newcommand{\bmp}{\begin{minipage}}
    \newcommand{\emp}{\end{minipage}}
    \newcommand{\btb}{\begin{tabular}}
    \newcommand{\etb}{\end{tabular}}
    \newcommand{\barr}{\begin{array}}
    \newcommand{\earr}{\end{array}}
    \newcommand{\bit}{\begin{itemize}}
    \newcommand{\eit}{\end{itemize}}
    \newcommand{\ben}{\begin{enumerate}}
    \newcommand{\een}{\end{enumerate}}
    \newcommand{\bct}{\begin{center}}
    \newcommand{\ect}{\end{center}}
    \newcommand{\bfr}{\begin{flushright}}
    \newcommand{\efr}{\end{flushright}}
    \newcommand{\bea}{\begin{eqnarray*}}
    \newcommand{\eea}{\end{eqnarray*}}
    \newcommand{\bqo}{\begin{quote}}
    \newcommand{\eqo}{\end{quote}}
    \newcommand{\bdc}{\begin{description}}
    \newcommand{\edc}{\end{description}}

    \newcommand{\Hom}{\mathrm{Hom}}
    
    \newcommand{\Ext}{\mathrm{Ext}}
    
    \newcommand{\Gal}{\mathrm{Gal}}
    \newcommand{\tr}{\mathrm{tr\,}}
    \newcommand{\ord}{\mathrm{ord\,}}

    \newcommand{\ind}{\mathrm{ind\,}}

    \newcommand{\res}{\mathrm{res\,}}

\input{amssym.def}
\input{amssym}
\input xy
\xyoption{all} \CompileMatrices

\newtheorem{theo}{Theorem}[section]
\newtheorem{prop}[theo]{Proposition}
\newtheorem{lem}[theo]{Lemma}
\newtheorem{cor}[theo]{Corollary}
\newtheorem{defi}[theo]{Definition}

\newtheorem{rem}[theo]{Remark}

\newcommand{\myfootnote}[1]{%
\renewcommand{\thefootnote}{}%
\footnotetext{#1}%
\renewcommand{\thefootnote}{\arabic{footnote}}%
}

\begin{document}
\xymatrixrowsep{3pc} \xymatrixcolsep{3pc}

\title{A conductor formula for completed group algebras}
\author{Andreas Nickel\thanks{The author acknowledges financial support provided by the DFG}}
\date{}

\maketitle

\begin{abstract}
    Let $\fo$ be the ring of integers in a finite extension of $\qp$. If $G$ is a finite group and $\Ga$ is a maximal $\fo$-order containing the group ring $\fo[G]$, Jacobinski's conductor formula
    gives a complete description of the central conductor of $\Ga$ into $\fo[G]$ in terms of characters of $G$. We prove a similar result for completed group algebras $\fo [[G]]$, where $G$
    is a $p$-adic Lie group of dimension $1$. We will also discuss several consequences of this result.
\end{abstract}

\section*{Introduction}
\myfootnote{{\it 2010 Mathematics Subject Classification:} 16H10, 16H20, 11R23}
\myfootnote{{\it Keywords:} central conductor, completed group algebras, extensions of lattices, Fitting invariants}
Let $\fo$ be the ring of integers in a number field $K$ (or in a finite extension $K$ of $\qp$)
and consider the group ring $\fo[G]$ of a finite group $G$ over $\fo$. The central conductor $\mc F(\fo[G])$ consists of all elements
$x$ in the center of $\fo[G]$ such that $x \Ga \subseteq\fo[G]$, where $\Ga \subseteq K[G]$ is a chosen maximal $\fo$-order containing $\fo[G]$, i.e.
$$\mc F(\fo[G]) = \left\{x \in \zeta(\fo[G])\mid x \Ga \subseteq\fo[G] \right\}.$$
Here, we write $\zeta(\La)$ for the center of a ring $\La$. A result of Jacobinski \cite{Jacobinski} (see also \cite[Theorem 27.13]{CR-I}) gives a complete description
of the central conductor in terms of the irreducible characters of $G$. More precisely, we have
\beq \label{eqn:central-conductor-J}
    \mc F (\fo[G]) = \bigoplus_{\chi} \frac{|G|}{\chi(1)} \fD\me(\fo[\chi] / \fo),
\eeq
where $\fD\me(\fo[\chi] / \fo)$ denotes the inverse different of $\fo[\chi]$, the ring of integers in
$K(\chi) := K(\chi(g)\mid g \in G)$, with respect to $\fo$, and the sum runs through all
absolutely irreducible characters of $G$ modulo the following Galois action:
If $\chi$ is an absolutely irreducible character of $G$ and $\si$ belongs to $\Gal(K(\chi)/K)$,
then $\si$ acts on $\chi$ as $^{\si}\chi(g) = \si(\chi(g))$
for every $g \in G$.
Jacobinski's main interest was in determining annihilators of $\Ext$; in fact, he showed that
$$\mc F(\fo[G]) \cdot \Ext^1_{\fo[G]} (M,N) = 0$$
for all $\fo[G]$-lattices $M$ and $\fo[G]$-modules $N$. For instance, it can be deduced from this result
that $|G| / \chi(1)$ annihilates $\Ext^1_{\fo[G]} (M_{\chi},N)$
if $M_{\chi}$ is an $\fo[G]$-lattice such that $K \otimes_{\fo} M_{\chi}$ is absolutely simple with character $\chi$. Later, Roggenkamp \cite{Roggenkamp}
showed that the annihilators obtained in this way are in fact the best possible in a certain precise sense.\\

In this article we consider completed group algebras $\fo[[G]]$, where $\fo$ denotes the ring of integers in a finite extension $K$ of $\qp$ and $G$ is a $p$-adic Lie group of
dimension $1$. Hence $G$ may be written as a semi-direct product $H \rtimes \Ga$ with finite $H$ and a cyclic pro-$p$-group $\Ga$, isomorphic to $\zp$. We will exclude the
special case $p=2$, as we will make heavily use of results of Ritter and Weiss \cite{towardII} (where the underlying prime is assumed to be odd)
on the total ring of fractions $\mc Q^K(G)$ of $\fo[[G]]$. However, it turns out that the results provided by Ritter and Weiss are not sufficient for our purposes
and so we shall have to determine the structure of $\mc Q^K(G)$ in more detail,
thereby generalizing and extending results of Lau \cite{Irene_Lau} (where $K = \qp$ and $G$ is pro-$p$).
We will do this in the first section.
In section $2$ we provide the necessary preparations for our main
theorem including results on reduced traces and conductors.
The main theorem will then be stated and proved in section $3$.
More precisely, if we define the central conductor in complete analogy to the group ring case, then we
have an inclusion
\begin{equation} \label{eqn:intro-cond}
    \bigoplus_{\chi / \sim} \frac{|H| w_{\chi}}{\chi(1)} \cdot \fD\me(\fo_{\chi} / \fo) \fo_{\chi}[[\Ga_{\chi}']] \subseteq \mc F(\fo[[G]]),
\end{equation}
where the sum runs through all absolutely irreducible characters of $G$ with open kernel up to a certain explicit equivalence relation.
Moreover, $w_{\chi}$ is the index of a certain subgroup (depending on $\chi$) in $G$ and $\fo_{\chi}$ denotes the ring of integers in
$K_{\chi} := K(\chi(h) \mid h \in H)$. Finally, $\Ga_{\chi}'$ is a cyclic pro-$p$-group which has
an explicitly determined topological generator. We will give precise definitions later in the text.
The inclusion (\ref{eqn:intro-cond}) is not far from being an equality. In fact, we show that its `$\chi$-part' becomes an equality
after localization at any height $1$ prime ideal of $\fo_{\chi}[[\Ga_{\chi}']]$ that does not contain $p$.
It is an equality if $G = H \times \Ga$ is a direct product or if no skewfields occur in the Wedderburn decomposition
of $\mc Q^K(G)$ (in fact it suffices to suppose that the Schur indices are not divisible by $p$). Moreover, we will also explicitly determine
the central conductor whenever $G$ is a pro-$p$-group; we use this to give an example where (\ref{eqn:intro-cond}) is a proper inclusion.\\

The proof of Jakobinski's central conductor formula does not carry over unchanged to the present situation for two reasons. First, the completed group algebra
is an order over the power series ring $\fo[[T]]$, but there is no canonical choice of embedding of $\fo[[T]]$ into $\zeta(\fo[[G]])$. Secondly and more seriously,
the ring $\fo[[T]]$ is a regular local ring, but it is not a Dedekind domain. Furthermore, even if we localize at a height one prime ideal,
the residue field will not be finite. Hence we do not have the well elaborated theory of maximal orders over discrete valuation rings with finite residue field
at our disposal. In the aforementioned two cases, when
$G = H \times \Ga$ is a direct product or when no skewfields occur in the Wedderburn decomposition,
we will overcome this problem by replacing our chosen maximal $\fo[[T]]$-order by a suitable maximal $\fo$-order.
When $G$ is a pro-$p$-group, it is the explicit description of the occurring skewfields due to Lau \cite{Irene_Lau} which allows us to
determine the central conductor.\\

Finally, we derive some consequences in section $4$. In particular, we obtain results for the corresponding $\Ext$-groups in analogy to the group ring case.
We also apply our main result to the theory of non-commutative Fitting invariants introduced by the author \cite{ich-Fitting}
and further developed in \cite{JN_Fitting}.
This theory may be applied
to $\fo[[G]]$-modules even if $G$ is non-abelian, but in contrast to the commutative case, the Fitting invariant of a finitely presented $\fo[[G]]$-module $M$
might not be contained in the annihilator of $M$. To obtain annihilators one has to multiply by a certain ideal $\mc H(\fo[[G]])$
of $\zeta(\fo[[G]])$ which is hard to determine in general.
However, it is easily seen that $\mc H(\fo[[G]])$ always contains the central conductor
so that our main theorem provides a method to compute explicit annihilators
of a finitely presented $\fo[[G]]$-module, at least, if we are able to compute its Fitting invariant.\\

{\sc Acknowledgement. \,}
The author is indebted to Henri Johnston for his many suggestions and helpful remarks.

\section{The total ring of fractions of a completed group algebra}

Let $p$ be an odd prime and let $G$ be a profinite group containing a finite normal subgroup $H$ such that
$G/H \simeq \Ga$ for a pro-p-group $\Ga$, isomorphic to $\zp$; thus $G$ can be written
as a semi-direct product $H \rtimes \Ga$ and is a $p$-adic Lie group of dimension $1$.
We denote the completed group algebra $\zp[[G]]$ by $\La(G)$ that is
\[
\Lambda(G) := \Z_{p}[[G]] = \varprojlim \Z_{p}[G/N],
\]
where the inverse limit is taken over all open normal subgroups $N$ of $G$.
If $K$ is a finite field extension of $\qp$
with ring of integers $\fo$, we put $\La^{\fo}(G) := \fo \otimes_{\zp} \La(G) = \fo[[G]]$.
We fix a topological
generator $\ga$ of $\Ga$ and choose a natural number $n$ such that $\ga^{p^n}$ is central in $G$.
Since we also have that $\Ga^{p^n} \simeq \zp$, there is an isomorphism $\fo [[\Ga^{p^n}]] \simeq \fo [[T]]$
induced by $\ga^{p^n} \mapsto 1+T$. Here, $R := \fo [[T]]$ denotes the power series ring in one variable over $\fo$.
If we view $\La^{\fo}(G)$ as an $R$-module (or more generally as a left $R[H]$-module), there is a decomposition
$$\La^{\fo}(G) = \bigoplus_{i=0}^{p^n-1} R[H] \ga^i.$$
Hence $\La^{\fo}(G)$ is finitely generated and free as an $R$-module and is an $R$-order in the separable $L:=Quot(R)$-algebra
$\mathcal Q^K (G) := L \otimes_R \La^{\fo}(G)$. Note that $\mathcal Q^K (G)$ is obtained from $\La^{\fo}(G)$ by inverting all regular elements.
By \cite[Lemma 1]{towardII} we have $\mathcal Q^K (G) = K \otimes_{\qp} \mc Q(G)$, where $\mc Q(G) := \mc Q^{\qp}(G)$
denotes the total ring of fractions of $\La(G)$.

Let $\qpc$ be an algebraic closure of $\qp$ and denote by $\irr(G)$ the set of absolutely irreducible $\qpc$-valued characters of $G$
with open kernel. Fix $\chi \in \irr(G)$.
Let $\eta$ be an irreducible constituent of $\res^{G}_H \chi$.
Then $G$ acts on $\eta$ as $\eta^g(h) = \eta(g\me hg)$ for $g \in G$, $h \in H$, and we set
$$St(\eta) := \{g \in G: \eta^g = \eta \}, \quad e(\eta) := \frac{\eta(1)}{|H|} \sum_{h \in H} \eta(h\me) h,
\quad e_{\chi} := \sum_{\eta \mid \res^G_H \chi} e(\eta).$$
Choose a finite Galois extension $E$ of $K$ such that the character $\chi$ has a realization
$V_{\chi}$ over $E$.
By \cite[Corollary to Proposition 6]{towardII} $e_{\chi}$ is a primitive central idempotent of $\mc Q^E (G)$.
In fact, it is shown that every primitive central idempotent of $\mc Q^c(G) := \qpc \otimes_{\qp} \mc Q(G)$ is an $e_{\chi}$,
and $e_{\chi} = e_{\chi'}$ if and only if $\chi = \chi' \otimes \rho$ for some character $\rho$ of $G$ of type $W$
(i.e.~$\res^G_H \rho = 1$).

By Clifford theory \cite[Proposition 11.4]{CR-I}
the irreducible constituents of $\res^{G}_H \chi$ are precisely the conjugates of $\eta$
under the action of $G$, each occurring with the same multiplicity $z_{\chi}$. By \cite[Lemma 4]{towardII}
we have $z_{\chi} = 1$ and thus equalities
\beq \label{eqn:e_chi}
    \res^{G}_H \chi = \sum_{i=0}^{w_{\chi}-1} \eta^{\ga^i},
    \quad e_{\chi} = \sum_{i=0}^{w_{\chi}-1} e(\eta^{\ga^i}) = \frac{\chi(1)}{|H| w_{\chi}} \sum_{h \in H} \chi(h\me) h,
\eeq
where $w_{\chi} := [G : St(\eta)]$. Note that $w_{\chi}$ is a power of $p$ as $H$ is a subgroup of $St(\eta)$.
We now put
$$K_{\chi} := K(\chi(h) \mid h \in H) \subseteq K(\eta) := K(\eta(h) \mid h \in H)$$
and note that $K_{\chi} = K_{\chi \otimes \rho}$ whenever $\rho$ is of type $W$.
As $g\me H g = H$, we have $K(\eta^g) = K(\eta)$ for every $g \in G$ and thus $K(\eta)$ does not depend on the
particular choice of irreducible constituent of $\res^G_H \chi$.

We let $\si \in \Gal(\qpc / K)$ act on $\chi$ as $^{\si}\chi(g) = \si(\chi(g))$ for all $g \in G$
and similarly on characters of $H$. Note that the actions on $\res^G_H \chi$ and $\eta$ factor through
$\Gal(K_{\chi}/K)$ and $\Gal(K(\eta)/K)$, respectively.

\begin{lem} \label{lem:field-index}
    The degree $[K(\eta):K_{\chi}]$ divides $w_{\chi}$; in particular, it is a power of $p$.
\end{lem}

\begin{proof}
    Fix $\si$ in $\Gal(K(\eta)/K_{\chi})$. As $\si$ acts trivially on $\res^G_H \chi$, the character $^{\si}\eta$
    also is an irreducible constituent of $\res^G_H \chi$. Hence there are irreducible constituents $\eta_1, \dots, \eta_s$
    of $\res^G_H \chi$ such that $\res^G_H \chi$ may be written as
    $$\res^G_H \chi = \sum_{i=1}^s \sum_{\si \in \Gal(K(\eta_i)/K_{\chi})} {^{\si}\eta_i}.$$
    Then $w_{\chi} \eta(1) = \chi(1) = \sum_{i=1}^s [K(\eta_i):K_{\chi}] \eta_i(1) = s [K(\eta):K_{\chi}] \eta(1)$
    as desired.
\end{proof}

By \cite[Proposition 5]{towardII} there is a unique element
$\ga_{\chi} \in \zeta(\mc Q^c (G)e_{\chi})$ such that $\ga_{\chi}$ acts trivially on $V_{\chi}$ and
$\ga_{\chi} = \ga^{w_{\chi}} \cdot c = c \cdot \ga^{w_{\chi}}$, where $c \in (\qpc[H] e_{\chi})\mal$.
An analysis of the proof in fact shows that $c \in (E[H] e_{\chi})\mal$ if both $\chi$ and $\eta$ have realizations over $E$.
We can and do assume this in the following.
Again by \cite[Proposition 5]{towardII} the element $\ga_{\chi}$ generates a procyclic $p$-subgroup $\Ga_{\chi}$
of $(\mc Q^E (G)e_{\chi})\mal$.
Moreover, $\ga_{\chi}$ induces an isomorphism $\mc Q^E (\Ga_{\chi})
\stackrel{\simeq}{\lto} \zeta(\mc Q^E (G)e_{\chi})$ by \cite[Proposition 6]{towardII}.

We let $\si \in \Gal(E/K)$ act on $\mc Q^E(G) = E \otimes_K \mc Q^K(G)$ as $\si(x \otimes y) = \si(x) \otimes y$ for all $x \in E$ and $y \in \mc Q^K(G)$.
We then have
\begin{equation} \label{eqn:unique-gammachi}
    \si(e_{\chi}) = e_{^\si{\chi}}, \quad \si(\ga_{\chi}) = \ga_{^{\si}{\chi}}
\end{equation}
for every $\si \in \Gal(E/K)$; here, the latter equality follows from the uniqueness of $\ga_{^{\si}{\chi}}$.
By (\ref{eqn:e_chi}) we have $\si(e_{\chi}) = e_{\chi}$ whenever $\si \in \Gal(E/K_{\chi})$;
in particular, we have an action of $\Gal(E/K_{\chi})$ on $\zeta(\mc Q^E (G)e_{\chi})$.

For a finite extension $F$ of $\qp$ let $U_F^1$ denote the group of principal units in $F$.

\begin{lem} \label{lem:principal unit}
  Fix $\chi \in \irr(G)$.
  Then there is a principal unit $x = x_{\chi} \in U_E^1$ with the following properties.
  \ben[(i)]
    \item
    The element $\ga_{\chi}' := x \ga_{\chi}$ is invariant under the action of $\Gal(E/K_{\chi})$ on $\zeta(\mc Q^E (G)e_{\chi})$.
    \item
    Let $m \geq 0$ be an integer. If $\ga^{p^m}$ acts trivially on $V_{\chi}$, then $x^{p^m}$ belongs to $U_{K_{\chi}}^1$.
  \een
  Moreover, if $\Ga_{\chi}'$ denotes the procyclic $p$-subgroup generated by $\ga_{\chi}'$, then
  $$\mc Q^E(\Ga_{\chi})^{\Gal(E/K_{\chi})} = \mc Q^{K_{\chi}}(\Ga_{\chi}').$$
\end{lem}

\begin{proof}
  We first observe that all claims do not depend on the choice of $E$. To see this let $F$ be a second finite Galois
  extension of $K$ such that the character $\chi$ can be realized over $F$. Replacing $F$ by $EF$ we may assume $E \subseteq F$.
  If there is an $x \in U_E^1$ with (i) and (ii), then we may use the same $x$ for $F$ as $U_E^1 \subseteq U_F^1$.
  Conversely, suppose there is an $x \in U_F^1$ which fulfills (i) and (ii) with $E$ replaced by $F$. Let $\si \in \Gal(F/E)$
  be arbitrary. Then $\si$ acts trivially on $\ga_{\chi}$ as $\ga_{\chi} \in \zeta(\mc Q^E (G)e_{\chi})$, and it acts trivially
  on $\ga_{\chi}'$ by (i). Hence $\si(x) = \si(\ga_{\chi}') \si(\ga_{\chi})\me = \ga_{\chi}' \ga_{\chi}\me = x$
  and thus $x \in U_F^1 \cap E = U_E^1$ as desired. Finally, as $\Gal(F/E)$ acts trivially on $\ga_{\chi}$,
  we have $\mc Q^F(\Ga_{\chi})^{\Gal(F/E)} = \mc Q^E(\Ga_{\chi})$, and thus the last statement of the Lemma
  does not depend on $E$ either.

  We now may assume that $E$ is obtained from $K_{\chi}$ by adjoining roots of unity.
  More precisely, the group $G / \ker (\chi)$ is finite as $\chi$ has open kernel,
  and by \cite[Theorem 15.16]{CR-I} we may take $E = K_{\chi}(\zeta)$,
  where $\zeta$ is a root of unity of order $|G / \ker (\chi)|$.
  Let us denote by $\mu_p(E)$ the group of all $p$-power roots of unity in $E$.
  Let $\si \in \Gal(E/K_{\chi})$ be arbitrary.
  Then $\res^G_H \chi = \res^G_H {^{\si}\chi}$ and thus ${^{\si}\chi} = \chi \otimes \rho_{\si}$ for some character $\rho_{\si}$ of type $W$.
  Hence
  $\si(\ga_{\chi}) = \ga_{^{\si}\chi} = \zeta_{\si} \ga_{\chi}$ with $\zeta_{\si} := \rho_{\si}(\ga)^{-w_{\chi}} \in \mu_p(E)$
  by \cite[Proposition 5]{towardII}.
  The assignment
  $$\Gal(E/K_{\chi}) \to \mu_p(E), \quad \si \mapsto \zeta_{\si} = \ga_{\chi}^{\si-1}$$
  is a $1$-cocycle.
  Let $E_0$ be the maximal unramified extension of $K_{\chi}$
  in $E$. Then $E = E_0(\mu_p(E))$ and thus $H^1(\Gal(E/E_0), \mu_p(E)) = 1$ by \cite[Proposition 9.1.6]{NSW}.
  By the inflation-restriction sequence this yields an
  isomorphism
  \beq \label{eqn:H1-iso}
    H^1(\Gal(E_0/K_{\chi}), \mu_p(E_0)) \simeq H^1(\Gal(E/K_{\chi}), \mu_p(E)).
  \eeq
  The natural map $H^1(\Gal(E_0/K_{\chi}), \mu_p(E_0)) \to H^1(\Gal(E_0/K_{\chi}), U_{E_0}^1)$
  is trivial, as the latter group vanishes by \cite[Proposition 7.1.2]{NSW}. This and (\ref{eqn:H1-iso})
  imply that there is an $a \in \mu_p(E) \cdot U_{E_0}^1 \subseteq U_E^1$
  such that $\zeta_{\si} = a^{\si-1}$ for all $\si \in \Gal(E/K_{\chi})$. We put $x := a\me$;
  then $\ga_{\chi}' := x \ga_{\chi}$ is easily seen to be invariant under the action of $\Gal(E/K_{\chi})$.

  Now let $m \geq 0$ be an integer such that $\ga^{p^m}$ acts trivially on $V_{\chi}$.
  Then also $c^{p^m} = (\ga_{\chi} \ga^{-w_{\chi}})^{p^m}$ acts trivially on $V_{\chi}$ and belongs to $(E[H]e_{\chi})\mal$.
  So we must have $c^{p^m} = e_{\chi}$.
  We deduce that $\ga_{\chi}^{p^m} = \ga^{p^m w_{\chi}} e_{\chi}$
  and thus $\si(\ga_{\chi}^{p^m}) = \ga_{\chi}^{p^m}$ for every $\si \in \Gal(E/K_{\chi})$.
  We obtain
  \begin{eqnarray*}
        \si(x^{p^m}) &  = & \si((\ga_{\chi}')^{p^m}) \si(\ga_{\chi}^{p^m})\me \\
            & = & (\ga_{\chi}')^{p^m} (\ga_{\chi}^{p^m})\me \\
            & = & x^{p^m};
  \end{eqnarray*}
  hence $x^{p^m}$ belongs to $U_E^1 \cap K_{\chi} = U_{K_{\chi}}^1$.

  For the last assertion, we observe that $U_E^1$ is a $\zp$-module and therefore
  \beq \label{eqn:hida}
    \La^{\fo_{E}}(\Ga_{\chi}) = \La^{\fo_E}(\Ga_{\chi}')
  \eeq
  as $(\ga_{\chi}')^z = x^z \ga_{\chi}^z$
  for every $z \in \zp$ (see also \cite[Lemma 1, p.199]{Hida_elementary}). Hence also $\mc Q^E(\Ga_{\chi}) = \mc Q^E(\Ga_{\chi}')$ and thus
  $$\mc Q^E(\Ga_{\chi})^{\Gal(E/K_{\chi})} = \mc Q^E(\Ga_{\chi}')^{\Gal(E/K_{\chi})} = \mc Q^{K_{\chi}}(\Ga_{\chi}')$$
  as desired.
\end{proof}

\begin{rem}
  The principal unit $x$ and the element $\ga_{\chi}'$ are unique up to a principal unit in $K_{\chi}$. To see this let $y \in U_E^1$
  be a second principal unit which fulfills (i) and (ii) of Lemma \ref{lem:principal unit} with $\ga_{\chi}'$
  replaced by $\ga_{\chi}'' := y \ga_{\chi}$. Then $yx\me = \ga_{\chi}'' (\ga_{\chi}')\me$ is a principal unit
  which is invariant under $\Gal(E/K_{\chi})$ and thus belongs to $U_{K_{\chi}}^1$.
\end{rem}

\example Assume that $G = H \times \Ga$ is a direct product. Then $w_{\chi} = 1$ and $\res^G_H \chi = \eta$ is irreducible.
If $\chi(\ga) = \chi(1)$, then $\ga_{\chi} = \ga e_{\chi}$ and we may choose $\ga_{\chi}' = \ga_{\chi}$. If $\chi(\ga) \not= \chi(1)$, we may write
$\chi = \chi' \otimes \rho$ with $\chi'(\ga) = \chi'(1) = \chi(1)$ and some character $\rho$ of type $W$. By \cite[Proposition 5]{towardII}
we have $\ga_{\chi} = \ga_{\chi'} \rho(\ga)\me = \rho(\ga)\me \ga e_{\chi}$. Then $x = \rho(\ga)$ fulfills (i) and (ii) of Lemma \ref{lem:principal unit}
and we find that $\ga_{\chi}' = \ga_{\chi'} = \ga e_{\chi} = \ga_{\chi'}'$.

\begin{defi}
  Let $\chi, \psi \in \irr(G)$. We say that $\chi$ and $\psi$ are equivalent over $K$ (and write $\chi \sim_{K} \psi$)
  if there is $\si \in \Gal(K_{\chi} / K)$ such that $^{\si}(\res^{G}_H \chi) = \res^{G}_H \psi$.
\end{defi}

Note that if $\chi \sim_K \psi$, then we have $K_{\chi} = K_{\psi}$ and via (\ref{eqn:unique-gammachi}) an isomorphism
$\mc Q^{K_{\chi}}(\Ga_{\chi}') \simeq \mc Q^{K_{\psi}}(\Ga_{\psi}')$. Moreover, we have
$\chi \sim_K \chi \otimes \rho$ whenever $\rho$ is a character of type $W$.

\begin{prop} \label{prop:centers}
Let $G$ be a $p$-adic Lie group of dimension $1$ and let $K$ be a finite extension of $\qp$.
Then there is an isomorphism
$$\zeta(\mc Q^K(G)) \simeq \bigoplus_{\chi \in \irr(G)/ \sim_K} \mc Q^{K_{\chi}}(\Ga_{\chi}'),$$
where the sum runs through all $\chi \in \irr(G)$ up to the above equivalence relation.
\end{prop}

\begin{proof}
Since there are only finitely many central primitive idempotents $e_{\chi}$ of $\mc Q^c(G)$, we may choose a finite Galois extension $E$ of $K$
such that $E[H]$ contains each $e_{\chi}$.
We will use the fact that the center of $\mc Q^K(G)$ coincides with the
$\Gal(E/K)$-invariants of $\zeta(\mc Q^E(G))$.
For this let $\si$ be an element of $\Gal(E/K)$. Fix $\chi, \psi \in \irr(G)$.
Then $\si(e_{\chi}) = e_{^{\si}\chi}$, and
$\si(e_{\chi}) = e_{\psi}$ if and only if $\res^{G}_H \psi = {^{\si}(\res^{G}_{H} \chi)}$.
In particular, the action of $\Gal(E/K)$ on $e_{\chi}$ factors through $\Gal(K_{\chi}/K)$, and we may write
\begin{equation} \label{eqn:def-of-ve}
    \zeta(\mc Q^K(G)) = \bigoplus_{\chi \in \irr(G)/ \sim_K} \zeta(\mc Q^K(G) \ve_{\chi}), \quad
    \ve_{\chi} := \sum_{\si \in \Gal(K_{\chi} / K)} \si(e_{\chi}) \in \zeta(\mc Q^K(G)).
\end{equation}
Now let $\beta \in \zeta(\mc Q^K(G) \ve_{\chi})$. We view $\beta$ as an element in $\zeta(\mc Q^E(G) \ve_{\chi})$ which is
invariant under $\Gal(E/K)$. We may therefore write
$$\beta = (\beta_{\si})_{\si} \in \bigoplus_{\si \in \Gal(K_{\chi} / K)} \zeta(\mc Q^E(G) \si(e_{\chi})) \simeq \bigoplus_{\si \in \Gal(K_{\chi} / K)} \mc Q^E(\Ga_{^{\si}\chi})$$
where, by abuse of notation, $\si$ denotes also a chosen lift of $\si$ in $\Gal(E/K)$.
As we have already mentioned before, the uniqueness of $\ga_{\chi}$ implies that $\ga_{^{\si}\chi} = \si(\ga_{\chi})$ for every $\si \in \Gal(E/K)$;
thus $\beta$ is determined by $\beta_1$, and
$\beta_1$ lies in $\mc Q^E(\Ga_{\chi})^{\Gal(E/K_{\chi})} = \mc Q^{K_{\chi}}(\Ga_{\chi}')$, where the equality is Lemma \ref{lem:principal unit}.
Now $\beta \mapsto \beta_1$ induces an isomorphism $\zeta(\mc Q^K(G) \ve_{\chi}) \simeq \mc Q^{K_{\chi}}(\Ga_{\chi}')$.
\end{proof}

\begin{rem}
In the special case, where $K = \qp$ and $G$ is a pro-$p$-group, a similar result has been established by Lau \cite{Irene_Lau}
using a different method. The same is true for Corollary \ref{cor:chi-equation} below.
\end{rem}

\begin{cor}
Let $K$ be a finite extension of $\qp$ with ring of integers $\fo$.
Choose a maximal $R$-order $\ti \La^{\fo}(G)$ containing $\La^{\fo}(G)$. Then
$$\zeta(\ti \La^{\fo}(G)) \simeq \bigoplus_{\chi \in \irr(G)/ \sim_K} \La^{\fo_{\chi}}(\Ga_{\chi}'),$$
where $\fo_{\chi}$ denotes the ring of integers in $K_{\chi}$.
\end{cor}

\begin{proof}
  This follows from Proposition \ref{prop:centers} as $\La^{\fo_{\chi}}(\Ga_{\chi}')$ is the integral closure of $R$
  in $\mc Q^{K_{\chi}}(\Ga_{\chi}')$.
\end{proof}

\begin{rem} \label{rem:embedding}
Here, $\ti \La^{\fo}(G)$ is an order over $R = \fo[[T]]$, where we have identified $1+T$ with $\ga^{p^n}$ for a chosen large $n$.
Let us fix a $\chi \in \irr(G)$ in each equivalence class over $K$.
If $m = m(\chi)$ is sufficiently large, then $\ga^{p^m}$ acts trivially on
$V_{\chi}$ and hence $\ga_{\chi}^{p^m} = (\ga^{p^m})^{w_{\chi}} e_{\chi}$.
Enlarging $n$ and $m$ if necessary (for the finitely many $\chi$), we may assume that $p^n = p^m \cdot w_{\chi}$.
Hence we may also assume that the inclusion
$$R = \fo[[T]] \into \La^{\fo_{\chi}}(\Ga_{\chi}')$$
is induced by $1 + T \mapsto \ga_{\chi}^{p^n / w_{\chi}} = x^{-p^m} (\ga_{\chi}')^{p^m} \in \La^{\fo_{\chi}}(\Ga_{\chi}')$,
where $x^{-p^m}$ belongs to $U_{K_{\chi}}^1$ by  Lemma \ref{lem:principal unit}.
We will fix such an $n$ for the rest of the paper.
\end{rem}

\begin{cor} \label{cor:chi-equation}
The algebra $\mc Q^K(G)$ has Wedderburn decomposition
$$\mc Q^K(G) \simeq  \bigoplus_{\chi \in \irr(G)/ \sim_K} (D_{\chi})_{n_{\chi} \times n_{\chi}},$$
where $n_{\chi} \in \N$ and $D_{\chi}$ is a skewfield with center $\mc Q^{K_{\chi}}(\Ga_{\chi}')$. If $s_{\chi}$ denotes the Schur index of $D_{\chi}$, then
we have an equality $\chi(1) = n_{\chi} s_{\chi}$.
\end{cor}

\begin{proof}
All assertions are immediate from Proposition \ref{prop:centers} apart from the last equality.
Let us denote the simple component $(D_{\chi})_{n_{\chi} \times n_{\chi}}$ by $A_{\chi}$.
With $E$ sufficiently large as in Proposition \ref{prop:centers} we compute
\bea
    (n_{\chi} s_{\chi})^2 & = & \dim_{\mc Q^{K_{\chi}}(\Ga_{\chi}')}(A_{\chi})\\
        & = & [K_{\chi} : K]\me \cdot \dim_{\mc Q^{K}(\Ga_{\chi}')}(A_{\chi})\\
        & {(i) \atop =} & [K_{\chi} : K]\me \cdot \dim_{\mc Q^{E}(\Ga_{\chi})}(E \otimes_K A_{\chi})\\
        & {(ii) \atop =} & [K_{\chi} : K]\me \cdot \sum_{\si \in \Gal(K_{\chi}/K)} \dim_{\mc Q^E(\Ga_{\chi})}(\mc Q^E(G) e_{^{\si}{\chi}})\\
        & {(iii) \atop =} & [K_{\chi} : K]\me \cdot \sum_{\si \in \Gal(K_{\chi}/K)} {^{\si}\chi(1)^2}\\
        & = & \chi(1)^2
\eea
Here, (i) follows from the equality $\mc Q^{E}(\Ga_{\chi}') = \mc Q^{E}(\Ga_{\chi})$ which was established
in the proof of Lemma \ref{lem:principal unit} (confer equation (\ref{eqn:hida})).
The isomorphism $E \otimes_K A_{\chi} \simeq \bigoplus_{\si \in \Gal(K_{\chi} / K)} \mc Q^E(G) e_{^{\si}{\chi}}$
implies (ii), and (iii)
is shown in the proof of \cite[Proposition 6]{towardII}.
\end{proof}

\begin{rem}
In the case $K=\qp$ and $G$ a pro-$p$-group one can determine the occurring skewfields explicitly
(see \cite[Theorem 1]{Irene_Lau} and also (\ref{eqn:skewfield-iso}) below).
\end{rem}

\begin{theo} \label{thm:splitting-field}
Fix $\chi \in \irr(G)$ and let
$A_{\chi} = (D_{\chi})_{n_{\chi} \times n_{\chi}}$ be the corresponding simple component of $\mc Q^K(G)$.
Then $E \otimes_K A_{\chi}$ splits whenever one (and hence every) irreducible constituent $\eta$ of $\res^G_H \chi$
can be realized over $E$. More precisely, we then have an isomorphism
$$E \otimes_K A_{\chi} \simeq \bigoplus_{\si \in \Gal(K_{\chi}/K)} \left( \mc Q^E(\Gamma_{\chi}') \right)_{\chi(1) \times \chi(1)}.$$
\end{theo}

\begin{proof}
Recall that $L = \mc Q^K(\Ga^{p^n})$ for our fixed sufficiently large $n$.
Let $E$ be a finite extension of $K$ and let $L' := E \otimes_K L = \mc Q^E(\Ga^{p^n})$.
As $L'$-vector space and more generally as left $L'[H]$-module, we have a decomposition
$$\mc Q^E(G) = \bigoplus_{i=0}^{p^n-1} L'[H] \ga^i.$$
Now fix $\chi \in \irr(G)$ and let $\eta$ be an irreducible constituent of $\res^G_H \chi$.
Suppose that $\eta$ has a realization over $E$. Recall the definition (\ref{eqn:def-of-ve}) of $\ve_{\chi}$.
As $K_{\chi} \subseteq K(\eta) \subseteq E$, we have a decomposition
$$E \otimes_K A_{\chi} = \mc Q^E(G) \ve_{\chi} = \bigoplus_{\si \in \Gal(K_{\chi} / K)} \mc Q^E(G) \si(e_{\chi}).$$
As the centers of $\mc Q^E(G) \si(e_{\chi})$, $\si \in \Gal(K_{\chi} / K)$ are isomorphic fields
via (\ref{eqn:unique-gammachi}) (compare also the proof of Proposition \ref{prop:centers}),
it suffices to show that $\mc Q^E(G)e_{\chi}$ splits.

Since $E$ is a subfield of $L'$ we have $L'[H] e(\eta) \simeq L'_{\eta(1) \times \eta(1)}$
and similarly for every other irreducible constituent of $\res^G_H \chi$. We obtain
\bea
    \mc Q^E(G)e_{\chi} & = & \bigoplus_{i=0}^{p^n-1} L'[H] e_{\chi} \ga^i\\
    & = & \bigoplus_{i=0}^{p^n-1} \bigoplus_{j=0}^{w_{\chi}-1} L'[H] e(\eta^{\ga^j}) \ga^i \\
    & \simeq & \bigoplus_{i=0}^{p^n-1} \bigoplus_{j=0}^{w_{\chi}-1} L'_{\eta(1) \times \eta(1)} \ga^i,
\eea
where we have used equation (\ref{eqn:e_chi}) for the second equality.
We now choose an indecomposable idempotent $f_{\eta} = f_{\eta} e(\eta)$ of $L'[H] e(\eta) \simeq L'_{\eta(1) \times \eta(1)}$.
Observe that for a second indecomposable idempotent $f'_{\eta}$ of $L'[H] e(\eta)$ we have an isomorphism $f_{\eta} L'[H] f'_{\eta} \simeq L'$. As $\mc Q^E(G)e_{\chi}$
is a simple algebra over its center $\mc Q^E(\Ga_{\chi}')$ by Corollary \ref{cor:chi-equation}, and $f_{\eta}$ is also an idempotent in $\mc Q^E(G)e_{\chi}$, it suffices to show
that $f_{\eta} \mc Q^E(G)e_{\chi} f_{\eta}$ is a field, namely $\mc Q^E(\Ga_{\chi}')$. For this, we first observe that
$f_{\eta,i} := \ga^i f_{\eta} \ga^{-i}$ is also
an indecomposable idempotent for every $0 \leq i < p^n$, and belongs to $L'[H] e(\eta^{\ga^i})$ as $\ga^i e(\eta) \ga^{-i} = e(\eta^{\ga^i})$ and $H$
is normal in $G$.
However, $e(\eta) = e(\eta^{\ga^i})$
if and only if $w_{\chi}$ divides $i$, and thus
$$f_{\eta} L'[H] f_{\eta,i} \simeq \left\{ \barr{ll} L' & \mbox{ if } w_{\chi} \mid i \\ 0 & \mbox{ otherwise.} \earr \right.$$
Since $e(\eta^{\ga^j}) \ga^i f_{\eta} = f_{\eta,i} e(\eta^{\ga^j}) \ga^i$ we therefore have
$$f_{\eta} \mc Q^E(G)e_{\chi} f_{\eta}  \simeq  \bigoplus_{i=0 \atop w_{\chi} \mid i}^{p^n-1} L' \ga^i
     =  \bigoplus_{i=0}^{w_{\chi}\me p^n-1} L' \ga^{w_{\chi} i}.$$
We conclude that $\dim_{L'}(f_{\eta} \mc Q^E(G)e_{\chi} f_{\eta}) = w_{\chi}\me p^n$.
However, recall that $L' = \mc Q^E(\Ga^{p^n})$ and $\ga^{p^n}$ identifies with $\ga_{\chi}^{p^n / w_{\chi}} = x^{-p^n/w_{\chi}} (\ga_{\chi}')^{p^n / w_{\chi}}$
by Remark \ref{rem:embedding}, where $x^{-p^n/w_{\chi}} \in U_{K_{\chi}}^1 \subseteq (L')\mal$. As $L'$-vector space we therefore have a decomposition
$$\mc Q^E(\Ga_{\chi}') = \bigoplus_{i=0}^{w_{\chi}\me p^n-1} L' (\ga_{\chi}')^i.$$
Thus we also have $\dim_{L'}(\mc Q^E(\Ga_{\chi}')) = w_{\chi}\me p^n$. Since $\mc Q^E(\Ga_{\chi}') = \zeta(\mc Q^E(G) e_{\chi})$ is contained in $f_{\eta} \mc Q^E(G)e_{\chi} f_{\eta}$,
we are done.
\end{proof}

\begin{cor}
There is a finite Galois extension $E$ of $K$ such that
$$\mc Q^E(G) \simeq \bigoplus_{\chi \in \irr(G)/ W} (\mc Q^E(\Ga_{\chi}))_{\chi(1) \times \chi(1)},$$
where the sum runs through all $\chi \in \irr(G)$ modulo twists with characters of type $W$.
In particular, no skewfields occur in the Wedderburn decomposition of the semi-simple algebra $\mc Q^E(G)$.
\end{cor}

\begin{proof}
  This is an immediate consequence of Theorem \ref{thm:splitting-field} and Corollary \ref{cor:chi-equation}
  once we observe that $\mc Q^E(\Ga_{\chi}) = \mc Q^E(\Ga_{\chi}')$ by (\ref{eqn:hida}) if $E$ is sufficiently large.
\end{proof}

\begin{cor} \label{cor:Schur-index}
  Write $\eta(1) = s_{\eta} n_{\eta}$, where $s_{\eta}$ denotes the Schur index of $\eta$.
  Then $s_{\chi}$ divides $s_{\eta} [K(\eta):K_{\chi}]$ and $n_{\eta}$ divides $n_{\chi}$.
\end{cor}

\begin{proof}
  By definition of the Schur index there is a field $E$ of minimal degree $s_{\eta}$ over $K(\eta)$ such that $\eta$ can be realized over $E$.
  However,
  \begin{eqnarray*}
    E \otimes_K A_{\chi} & \simeq & E \otimes_K \mc Q^{K_{\chi}}(\Gamma_{\chi}') \otimes_{\mc Q^{K_{\chi}}(\Gamma_{\chi}')} A_{\chi} \\
        & \simeq & \bigoplus_{\sigma \in \Gal(K_{\chi} / K)} E \otimes_{K_{\chi}} \mc Q^{K_{\chi}}(\Gamma_{\chi}')
            \otimes_{\mc Q^{K_{\chi}}(\Gamma_{\chi}')} A_{\chi}\\
        & \simeq & \bigoplus_{\sigma \in \Gal(K_{\chi} / K)} \mc Q^{E}(\Gamma_{\chi}') \otimes_{\mc Q^{K_{\chi}}(\Gamma_{\chi}')} A_{\chi}
  \end{eqnarray*}
  splits by Theorem \ref{thm:splitting-field}. Now \cite[Theorem 28.5]{MO} implies that $s_{\chi}$ divides
  $$[\mc Q^{E}(\Gamma_{\chi}'):\mc Q^{K_{\chi}}(\Gamma_{\chi}')] = [E:K_{\chi}] = s_{\eta} [K(\eta):K_{\chi}].$$
  Moreover, we have
  $$w_{\chi} s_{\eta} n_{\eta} = w_{\chi} \eta(1) = \chi(1) = n_{\chi} s_{\chi} \mid n_{\chi} s_{\eta} [K(\eta):K_{\chi}] \mid n_{\chi} s_{\eta} w_{\chi},$$
  where the last divisibility is due to Lemma \ref{lem:field-index}. Thus $n_{\eta}$ divides $n_{\chi}$ as claimed.
\end{proof}

\example Assume that $H$ is a $p$-group. Then $s_{\eta} = 1$ for all irreducible characters $\eta$ of $H$ by a result of Roquette \cite{Roquette}.
However, Lau \cite{Irene_Lau} gives examples, where $s_{\chi} = [K(\eta):K_{\chi}]$ is non-trivial. \\

Let us denote the global dimension of a ring $\La$ by $\gldim (\La)$.
The following lemma is only needed to justify Remark \ref{rem:condition-holds} below and can be skipped
on a first reading.

\begin{lem} \label{lem:first_technical}
Let $A$ be a separable $L$-algebra of finite dimension over $L$ and
let $\La$ be a maximal $R$-order in $A$.
If $A$ is split, then there is a maximal $\fo$-order $\De$ in $\La$ such that $R \otimes_{\fo} \De = \La$
(and thus $L \otimes_{\fo} \De = A$).
\end{lem}

\begin{proof}
There are natural numbers $k>0$ and $n_i$, $1\leq i \leq k$ such that $A = \bigoplus_{i=1}^k A_i$ and $A_i \simeq L_{n_i \times n_i}$.
We put $\ti\La_i :=  R_{n_i \times n_i}$ and $\ti\La := \bigoplus_{i=1}^k \ti\La_i$.
Observe that this is a maximal $R$-order in $A$ by \cite[Theorem 8.7]{MO}.
Then $\ti\De := \bigoplus_{i=1}^k \fo_{n_i \times n_i}$ is a maximal $\fo$-order in $\bigoplus_{i=1}^k K_{n_i \times n_i}$ and has
$R \otimes_{\fo} \ti\De = \ti\La$ and $L \otimes_{\fo} \ti\De = A$.
Since global dimension is invariant under Morita equivalence (cf.~\cite[Corollary, p.~476]{Ramras_MO}), we have
\[
    \gldim(\ti\La_i ) = \gldim(R) = 2.
\]
Moreover, $\ti\La_i / \rad(\ti\La_i)$ is a matrix ring over the residue field of $R$ and thus a simple artinian ring.
As likewise $\La$ decomposes into $\La = \bigoplus_{i=1}^k \La_i$, where each $\La_i$ is a maximal $R$-order in $A_i$,
we may apply \cite[Theorem 5.4]{Ramras_MO} in each simple component: there is a unit $a_i \in A_i$ such that $\La_i = a_i\me \ti\La_i a_i$.
We put $a := \sum_{i=1}^k a_i$ so that $\La = a\me \ti\La a$. We further put $\De := a\me \ti\De a$. Then $\De$ is a maximal $\fo$-order in $\La$,
and we have
$$R \otimes_{\fo} \De = a\me (R \otimes_{\fo} \ti\De) a = a\me \ti\La a = \La$$
as desired.
\end{proof}

\section{Traces and conductors}

Let $R$ be a noetherian integrally closed domain with quotient field $L$. If $A$ is a simple $L$-algebra, we denote by $\tr_{A/L}$
the reduced trace from $A$ to $L$. If $A$ is separable, then by \cite[Theorem 9.26]{MO}
the reduced trace gives rise to a symmetric associative nondegenerate bilinear form
$A \times A \to L$ which sends $(a,b)$ to $\tr_{A/L}(ab)$. Note that if $A=L'$ is a field, then $\tr_{L'/L} = \Tr_{L'/L}$
is the ordinary trace of fields. Now let $\La$ be an $R$-order in $A$. By a $\La$-lattice we mean a finitely generated $\La$-module
which is torsionfree as $R$-module. We consider duality
with respect to the trace form: For each full left $\La$-lattice $M$ in $A$ (that is $A \otimes_{\La} M = A$) we associate the dual right $\La$-lattice
$$D(M/R) := \left\{x \in A \mid \tr_{A/L}(xM) \subseteq R \right\}.$$
In particular, applying this construction to the two-sided $\La$-lattice $\La$
we obtain a two-sided $\La$-lattice $D(\La/R)$ which (by abuse of language) we will call the \emph{inverse different} of $\La$
with respect to $R$. If $D(\La/R)$ happens to be invertible, we call $\fD(\La/R) := D(\La/R)\me$ the \emph{different} of $\La$
with respect to $R$.

\begin{lem} \label{lem:product_differents}
    Let $R$ be a noetherian integrally closed domain with quotient field $L$ and let $A$ be a simple separable $L$-algebra.
    Let $L \subseteq L' \subseteq \zeta(A)$ be fields and denote the integral closure of $R$ in $L'$ by $R'$.
    Then for every maximal $R$-order $\La$ in A, we have
    $$D(\Lambda/R) \supseteq D(\La/R') D(R'/R)$$
    with equality if $D(R'/R)$ is invertible.
\end{lem}

\begin{proof}
  As $\La$ is also a maximal $R'$-order by \cite[Theorem 10.5]{MO},
  we have $\La = R' \La$ and $D(\La/R')$ is defined.
  Let $a \in D(\La/R')$ and $b \in D(R'/R)$ be arbitrary. Then
  $$\tr_{A/L}(ab \La) = \tr_{L'/L}(\tr_{A/L'}(ab \La)) = \tr_{L'/L}(b \tr_{A/L'}(a \La)) \subseteq \tr_{L'/L}(b R') \subseteq R$$
  as desired. Now assume that $D(R'/R) = \fD(R'/R)\me$ is invertible. We conclude
  \bea
    x \in D(\La/R) & \iff & \tr_{L'/L}(\tr_{A/L'}(x \La)) \subseteq R\\
    & \iff & \tr_{L'/L}(\tr_{A/L'}(x R' \La)) \subseteq R\\
    & \iff & \tr_{L'/L}(R' \tr_{A/L'}(x \La)) \subseteq R\\
    & \iff & \tr_{A/L'}(x \La) \subseteq D(R'/R)\\
    & \iff & \tr_{A/L'}(x \fD(R'/R)\La) \subseteq R'\\
    & \iff & x \fD(R'/R) \subseteq D(\La/R')\\
    & \iff & x \in D(\La/R') D(R'/R).
  \eea
\end{proof}

Now let $A$ be a semisimple $L$-algebra and let $\Tr$ be the ordinary trace map from $A$ to $L$. For every $x \in A$
we associate the homomorphism $\Tr_x \in \Hom_L(A,L)$ defined by $\Tr_x(a) := \Tr(xa)$ for all $a \in A$.
Let $\La$ be an $R$-order in $A$. For a full left $\La$-lattice $M$ in $A$ we may also consider duality with respect to the
ordinary trace form:
$$D_{\ord}(M/R) := \left\{x \in A \mid \Tr(xM) \subseteq R \right\} = \left\{x \in A \mid \Tr_x(M) \subseteq R \right\}$$
is a right $\La$-lattice in $A$. We have a canonical homomorphism of right $\La$-modules
\beq \label{eqn:dual-to-dual}
    \de_M: D_{\ord}(M/R)  \to  M^+ := \Hom_R(M,R), \quad
    x \mapsto \Tr_x|_M,
\eeq
where $\la \in \La$ acts on $f \in M^+$ as $f^{\la}(m) := f(\la m)$ for all $m \in M$.
Similar observations hold for full right $\La$-lattices in $A$.

Recall that an $R$-module $M$ is called {\it reflexive} if the canonical map $M \to M^{++}$, $m \mapsto \left[f \mapsto f(m)\right]$
is an isomorphism.

\begin{prop} \label{prop:double-dual}
    Let $R$ be a noetherian integrally closed domain with quotient field $L$. Let $A$ be a semisimple $L$-algebra and let $\La$
    be an $R$-order in $A$.
    If the ordinary trace $\Tr$ gives rise to a nondegenerate bilinear form $A \times A \to L$, $(a,b) \mapsto \Tr(ab)$,
    then the homomorphism $\de_M$ in (\ref{eqn:dual-to-dual}) is an isomorphism for every full $\La$-lattice $M$ in $A$. In particular,
    $D_{\ord}(D_{\ord}(M/R)/R) = M$ if and only if $M$ is reflexive.
\end{prop}

\begin{proof}
  As the bilinear form $A \times A \to L$, $(a,b) \mapsto \Tr(ab)$ is nondegenerate, we have $\Tr_x = \Tr_y$
  if and only if $x=y$. It follows that the map $\de_A: A \to \Hom_L(A,L)$, $x \mapsto \Tr_x$ is injective
  and thus an isomorphism of $L$-vector spaces as both sides have the same dimension over $L$.
  It is also a homomorphism of right $A$-modules.
  Now let $M$ be a full left $\La$-lattice in $A$. Then $\de_M$ is just the restriction of $\de_A$
  to $D_{\ord}(M/R)$ and thus injective. However, by \cite[Remark p.268]{NSW} restriction to $M$ yields an isomorphism
  \bea
    M^+ & \simeq & \left\{\phi \in \Hom_L(A,L) \mid \phi(M) \subseteq R \right\}\\
    & = & \left\{\Tr_x \mid x \in A, ~ \Tr_x(M) \subseteq R \right\}\\
    & = & \left\{\Tr_x \mid x \in D_{\ord}(M/R) \right\}.
  \eea
  Hence the image of $\de_M$ is $M^+$ as claimed.
\end{proof}

We now return to the case $L = Quot(R)$, where $R = \fo[[T]]$ is the power series ring in one variable over the ring of integers $\fo$ in $K$.
Recall that as an $L$-vector space we have $\mc Q^K(G) = \bigoplus_{i=0}^{p^n-1} L[H] \ga^i = \bigoplus_{i=0}^{p^n-1} L \ga^i[H]$.
We denote by $\Tr$ the ordinary trace map from $\mc Q^K(G)$ to $L$.

\begin{lem} \label{lem:trace-dualbasis}
The elements $\ga^i h$, $0\leq i < p^n$, $h \in H$ form an $L$-basis of $\mc Q^K(G)$ and
$$\Tr(\ga^i h) = \left\{ \barr{ll} p^n |H| & \mbox{ if } \ga^i = h = 1\\ 0 & \mbox{ otherwise.}\earr\right.$$
Its dual basis with respect to $\Tr$ is given by $(p^n|H|)\me h\me \ga^{-i}$, $0\leq i < p^n$, $h \in H$.
\end{lem}
\begin{proof}
It is clear that these elements form an $L$-basis of $\mc Q^K(G)$.
Now let $0 \leq i < p^n$ and $h \in H$. We use the same basis to compute $\Tr(\ga^i h)$. For every $0 \leq j < p^n$ and $h' \in H$, we have to write
$\ga^i h \ga^j h'$ as an $L$-linear combination of these basis elements. However,
$\ga^i h \ga^j h' = \ga^{i+j} h_j h'$ with $h_j := \ga^{-j} h \ga^j \in H$.
This actually belongs to this basis if $i+j < p^n$; if $i+j \geq p^n$, then $\ga^{i+j} h_j h' = \ga^{p^n} \ga^{i+j-p^n} h_j h'$ with $\ga^{p^n} \in L = \mc Q^K(\Ga^{p^n})$. Since only the diagonal entries of the corresponding $p^n |H| \times p^n |H|$ matrix contribute to $\Tr(\ga^i h)$,
we now suppose that $\ga^j h' = \ga^{i+j} h_j h'$ if $i+j<p^n$, and that $\ga^j h' = \ga^{i+j-p^n} h_j h'$ if $i+j \geq p^n$.
In both cases we must have $h_j = 1$ and thus $h=1$. In the first case, we also have $j=i+j$ which forces $i=0$, whereas in the second
case $j = i+j-p^n$ which gives $i = p^n$. This contradicts $i<p^n$ and so the latter case does not occur.
We see that this matrix has all diagonal entries equal to zero if $\ga^i h \not= 1$;
hence $\Tr(\ga^i h) = 0$ in this case. That $\Tr(1) = p^n |H|$ is clear.
It is now easily checked that $(p^n|H|)\me h\me \ga^{-i}$, $0\leq i < p^n$, $h \in H$ is the dual basis.
\end{proof}

\begin{lem} \label{lem:different_welldef}
Let $K'$ be a finite field extension of $K$ with ring of integers $\fo'$ and let $a$ be a principal unit in some finite extension of $K'$.
Let $n \in \N_0$ and suppose that $a^{p^n}$ belongs to $U_{K'}^1$. Consider $\La^{\fo'}(\Ga)$ as $R$-order via the embedding
$$\iota_{a,n}: R \into \La^{\fo'}(\Ga), \quad 1+T \mapsto (a \ga)^{p^n}.$$
Then the inverse different
$D(\La^{\fo'}(\Ga) / R) = \left\{x \in \mc Q^{K'}(\Ga) \mid \Tr_{\mc Q^{K'}(\Ga) / L}(x \La^{\fo'}(\Ga)) \subseteq R \right\}$
is given by
$$D(\La^{\fo'}(\Ga) / R) = p^{-n} \fD\me(\fo'/ \fo) \La^{\fo'}(\Ga),$$
where $\fD\me(\fo'/ \fo)$ denotes the usual inverse different of $\fo'$ with respect to $\fo$.
In particular, $D(\La^{\fo'}(\Ga) / R)$ is invertible.
\end{lem}

\begin{proof}
First, we consider the case $a=1$ and $n=0$.
If $x_1, \dots, x_k$ form an $\fo$-basis of $\fo'$,
then $x_1,\dots,x_k$ are also an $\fo[[T]]$-basis of $\fo'[[T]]$ which is isomorphic to $\La^{\fo'}(\Ga)$ via $1+T \mapsto \ga$. Hence its dual basis with respect to the ordinary
trace $\Tr_{K'/K}$ of fields, is also a dual basis with respect to $\Tr_{\mc Q^{K'}(\Ga) / L}$.
Hence $D(\La^{\fo'}(\Ga) / R)$ equals $\fD\me(\fo'/ \fo) \La^{\fo'}(\Ga)$ and is invertible in this case.
For the general case let $\iota_{a,n}': \La^{\fo'}(\Ga) \into \La^{\fo'}(\Ga)$ be induced from $\ga \mapsto (a \ga)^{p^n}$.
Then $\iota_{a,n} = \iota_{a,n}' \circ \iota_{1,0}$. We may therefore assume $\fo = \fo'$ by
the first part of the proof and an application of Lemma \ref{lem:product_differents}.
The inverse different only depends upon the image of $R$ in $\La^{\fo}(\Ga)$. However, $\iota_{a,n}(R) = \iota_{1,n}(R)$
as $a^{p^n} \in U_{K'}^1$. Thus we may also assume $a=1$. But in this case the result follows
from Lemma \ref{lem:trace-dualbasis} with $G = \Ga$.
\end{proof}

By Corollary \ref{cor:chi-equation} we may write $\mc Q^K(G) = \bigoplus_{\chi \in \irr(G)/ \sim_K} A_{\chi}$,
where $A_{\chi} = (D_{\chi})_{n_{\chi} \times n_{\chi}}$, $n_{\chi} \in \N$ and $D_{\chi}$ is a skewfield
with Schur index $s_{\chi}$ and center $\mc Q^{K_{\chi}}(\Ga_{\chi}')$. Again by Corollary \ref{cor:chi-equation}, we have $\chi(1) = s_{\chi} n_{\chi}$.
Hence the ordinary trace may be written as
\beq \label{eqn:trace}
    \Tr = \sum_{\chi \in \irr(G)/ \sim_K} \chi(1) \tr_{\chi},
\eeq
where $\tr_{\chi}$ denotes the reduced trace from $A_{\chi}$ to $L$. Moreover, we have
\beq \label{eqn:trace_composition}
    \tr_{\chi} = \Tr_{\mc Q^{K_{\chi}}(\Ga_{\chi}') / L} \circ \tr_{A_{\chi} / \mc Q^{K_{\chi}}(\Ga_{\chi}')},
\eeq
where $\Tr_{\mc Q^{K_{\chi}}(\Ga_{\chi}') / L}$ denotes the ordinary trace of fields and $\tr_{A_{\chi} / \mc Q^{K_{\chi}}(\Ga_{\chi}')}$
denotes the reduced trace from $A_{\chi}$ into its center. Recall from Remark \ref{rem:embedding} that we have fixed a sufficiently large integer $n \geq 0$
such that $R = \fo[[T]]$ embeds into
$\La^{\fo_{\chi}}(\Ga_{\chi}')$ via $1+T \mapsto (x\me \ga_{\chi}')^{p^n / w_{\chi}}$. Now Lemma \ref{lem:product_differents} and
Lemma \ref{lem:different_welldef} (with the embedding
$\iota_{a, m}$, where $a = x^{-1}$ and $p^m = p^n/w_{\chi}$) imply that the following definition does not depend on the choice of the embedding.

\begin{defi}
Choose a maximal $R$-order $\ti \La^{\fo}(G)$ containing $\La^{\fo}(G)$. We have a decomposition $\ti \La^{\fo}(G) = \bigoplus_{\chi \in \irr(G)/ \sim_K} \ti \La^{\fo}_{\chi}(G)$,
where each $\ti \La^{\fo}_{\chi}(G)$ is a maximal $R$-order in $A_{\chi}$.
For sufficiently large $n$ we call the two-sided $\ti \La^{\fo}_{\chi}(G)$-lattice
$$D_{\chi}(\ti \La^{\fo}(G)) = D_{\norm}(\tilde\La^{\fo}_{\chi}(G) / R) := p^n D(\tilde\La^{\fo}_{\chi}(G) / R)
 = p^n \cdot \left\{x \in A_{\chi}\mid \tr_{\chi}(x \tilde\La^{\fo}_{\chi}(G)) \subseteq R \right\}$$
the {\it normalized inverse different}.
\end{defi}


\begin{defi}
Let $\La \subseteq\ti\La$ be a pair of rings. Then
$$(\ti\La: \La)_l := \left\{x \in \ti\La \mid x \ti\La \subseteq\La \right\}$$
is called the {\it left conductor} of $\ti\La$ into $\La$. Similarly,
$$(\ti\La: \La)_r := \left\{x \in \ti\La \mid \ti\La x \subseteq\La \right\}$$
is called the {\it right conductor} of $\ti\La$ into $\La$.
\end{defi}

\begin{prop} \label{prop:conductor-reflexive}
  Let $\ti \La^{\fo}(G)$ be a maximal $R$-order containing $\La^{\fo}(G)$.
  Then $\La^{\fo}(G)$, the maximal order $\ti \La^{\fo}(G)$ and the left and right conductor
  $(\ti \La^{\fo}(G) : \La^{\fo}(G))_l$ and $(\ti \La^{\fo}(G) : \La^{\fo}(G))_r$ are reflexive (and thus free) $R$-modules.
\end{prop}

\begin{proof}
  We first observe that $R$ is a $2$-dimensional regular local ring and thus every reflexive $R$-module is free (and vice versa)
  by \cite[Proposition 5.1.9]{NSW}. As $\La^{\fo}(G)$ is free over $R$, it is also reflexive. Moreover, every maximal $R$-order
  is reflexive by \cite[Theorem 11.4]{MO}. Finally, let $M := (\ti \La^{\fo}(G) : \La^{\fo}(G))_r$
  (the argument for $(\ti \La^{\fo}(G) : \La^{\fo}(G))_l$ is similar).
  We denote by $P(R)$ the set of prime ideals of $R$ of height $1$.
  As $M$ is $R$-torsionfree, the proof of \cite[Proposition 5.1.8]{NSW} shows that
  we have an injection $M \into M^{++} = \bigcap_{\fp \in P(R)} M_{\fp}$, where the equality is \cite[Lemma 5.1.2]{NSW}.
  Now let $x \in \bigcap_{\fp \in P(R)} M_{\fp}$ be arbitrary; we have to show that $x \in M$.
  For every $\fp \in P(R)$ we have $x \in \ti \La^{\fo}(G)_{\fp}$ and $\ti \La^{\fo}(G)_{\fp} x \subseteq \La^{\fo}(G)_{\fp}$.
  Thus $x \in \bigcap_{\fp \in P(R)} \ti \La^{\fo}(G)_{\fp} = \ti \La^{\fo}(G)$ and
  $$\ti \La^{\fo}(G) x  = \bigcap_{\fp \in P(R)} \ti \La^{\fo}(G)_{\fp} x \subseteq \bigcap_{\fp \in P(R)} \La^{\fo}(G)_{\fp} = \La^{\fo}(G),$$
  that is $x \in M$.
\end{proof}

We now establish the following analogue of Jacobinski's conductor formula \cite[Theorem 27.8]{CR-I}.

\begin{theo} \label{thm:global_conductor}
Let $\ti \La^{\fo}(G)$ be a maximal $R$-order containing $\La^{\fo}(G)$. Then
$$(\ti \La^{\fo}(G) : \La^{\fo}(G))_l = (\ti \La^{\fo}(G) : \La^{\fo}(G))_r =
    \bigoplus_{\chi \in \irr(G)/ \sim_K} \frac{|H|}{\chi(1)} D_{\chi}(\ti \La^{\fo}(G)).$$
\end{theo}

\begin{proof}
  Put $M := (\ti \La^{\fo}(G) : \La^{\fo}(G))_r$; then $M$ is the largest left $\ti \La^{\fo}(G)$-lattice in $\La^{\fo}(G)$.
  So its ordinary dual $D_{\ord}(M/R)$ is the smallest right $\ti \La^{\fo}(G)$-lattice containing $D_{\ord}(\La^{\fo}(G)/R)$.
  Since $\La^{\fo}(G)$ is free over $R$ with basis $\ga^i h$, $0 \leq i < p^n$, $h \in H$, we therefore have
  $$D_{\ord}(M/R) = D_{\ord}(\La^{\fo}(G)/R) \ti \La^{\fo}(G) = (p^n |H|)\me \ti \La^{\fo}(G),$$
  where we have used Lemma \ref{lem:trace-dualbasis} for the second equality. However, $M$ is reflexive by Proposition \ref{prop:conductor-reflexive}
  and so Proposition \ref{prop:double-dual} yields
  $$M = D_{\ord}(D_{\ord}(M/R)/R) = D_{\ord}((p^n |H|)\me \ti \La^{\fo}(G)/R) = p^n |H| D_{\ord}(\ti \La^{\fo}(G)/R).$$
  The theorem now follows from the equalities
  \begin{eqnarray*}
    D_{\ord}(\ti \La^{\fo}(G)/R) & = & \bigoplus_{\chi \in \irr(G)/ \sim_K} D_{\ord}(\ti \La^{\fo}_{\chi}(G)/R) \\
    D_{\ord}(\ti \La^{\fo}_{\chi}(G)/R) & = & \chi(1)\me D(\ti \La^{\fo}_{\chi}(G)/R) \\
        & = & (\chi(1)p^n)\me D_{\chi}(\ti \La^{\fo}(G)),
  \end{eqnarray*}
  as a similar argument works for $(\ti \La^{\fo}(G) : \La^{\fo}(G))_l$.
\end{proof}

\section{A formula for the central conductor}

\begin{defi}
Let $\ti\La^{\fo}(G)$ be a maximal $R$-order containing $\La^{\fo}(G)$. Then the central conductor of $\ti\La^{\fo}(G)$ into $\La^{\fo}(G)$ is defined to be
\bea
    \mc F(\La^{\fo}(G)) = \mc F(\ti\La^{\fo}(G) / \La^{\fo}(G)) & := & \zeta(\La^{\fo}(G)) \cap (\ti\La^{\fo}(G) : \La^{\fo}(G))_l \\
        & = & \zeta(\La^{\fo}(G)) \cap (\ti\La^{\fo}(G) : \La^{\fo}(G))_r \\
        & = & \left\{x \in \zeta(\La^{\fo}(G)) \mid x \ti\La^{\fo}(G) \subseteq\La^{\fo}(G)\right\}.
\eea
\end{defi}

\begin{rem}
We will see below that the central conductor only depends on $\La^{\fo}(G)$ and not on the choice of maximal order containing $\La^{\fo}(G)$.
Hence the notation $\mc F(\La^{\fo}(G))$ is justified.
\end{rem}

By Corollary \ref{cor:chi-equation} we may write $\mc Q^K(G) = \bigoplus_{\chi \in \irr(G)/ \sim_K} A_{\chi}$, where each $A_{\chi} \simeq (D_{\chi})_{n_{\chi} \times n_{\chi}}$ is simple.
Similarly, $\ti\La^{\fo}(G)$ decomposes into $\bigoplus_{\chi \in \irr(G)/ \sim_K} \ti\La^{\fo}_{\chi}(G)$, where each $\ti\La^{\fo}_{\chi}(G)$ is a maximal $R$-order in $A_{\chi}$
with center $\La^{\fo_{\chi}}(\Ga_{\chi}')$.
The central conductor is an ideal in $\zeta(\ti\La^{\fo}(G)) = \bigoplus_{\chi \in \irr(G)/ \sim_K} \La^{\fo_{\chi}}(\Ga_{\chi}')$ and
so we may write
$$\mc F(\ti\La^{\fo}(G) / \La^{\fo}(G)) = \bigoplus_{\chi \in \irr(G)/ \sim_K} \mc F_{\chi}(\ti\La^{\fo}(G) / \La^{\fo}(G)),$$
where each $\mc F_{\chi}(\ti\La^{\fo}(G) / \La^{\fo}(G))$
is an ideal in $\La^{\fo_{\chi}}(\Ga_{\chi}')$.
We define
$$d_{\chi} := \mc Q^{K_{\chi}}(\Ga_{\chi}') \cap D(\ti \La^{\fo}_{\chi}(G)/ \La^{\fo_{\chi}}(\Ga_{\chi}')).$$
We denote by $P_{\chi}$ the set of prime ideals of $\La^{\fo_{\chi}}(\Ga_{\chi}')$ of height $1$.

\begin{lem} \label{lem:independence}
  The fractional $\La^{\fo_{\chi}}(\Ga_{\chi}')$-ideal $d_{\chi}$ is a reflexive $\La^{\fo_{\chi}}(\Ga_{\chi}')$-module which
  does not depend on the choice of maximal order $\ti\La^{\fo}(G)$.
\end{lem}

\begin{proof}
  We first show that $d_{\chi}$ is reflexive.
  As it is torsionfree, we have an injection $d_{\chi} \into d_{\chi}^{++} = \bigcap_{\fp \in P_{\chi}} (d_{\chi})_{\fp}$.
  Now let $x \in \bigcap_{\fp \in P_{\chi}} (d_{\chi})_{\fp}$. Then clearly $x \in \mc Q^{K_{\chi}}(\Ga_{\chi}')$
  and
  $$\tr_{A_{\chi} / \mc Q^{K_{\chi}}(\Ga_{\chi}')}(x \ti \La^{\fo}_{\chi}(G)_{\fp}) \subseteq \La^{\fo_{\chi}}(\Ga_{\chi}')_{\fp}$$
  for all $\fp \in P_{\chi}$. This implies
  \begin{eqnarray*}
    \tr_{A_{\chi} / \mc Q^{K_{\chi}}(\Ga_{\chi}')}(x \ti \La^{\fo}_{\chi}(G)) & \subseteq &
        \bigcap_{\fp \in P_{\chi}} \tr_{A_{\chi} / \mc Q^{K_{\chi}}(\Ga_{\chi}')}(x \ti \La^{\fo}_{\chi}(G)_{\fp}) \\
        & \subseteq & \bigcap_{\fp \in P_{\chi}} \La^{\fo_{\chi}}(\Ga_{\chi}')_{\fp} \\
        & = & \La^{\fo_{\chi}}(\Ga_{\chi}').
  \end{eqnarray*}
  Thus $x$ belongs to $\mc Q^{K_{\chi}}(\Ga_{\chi}') \cap D(\ti \La^{\fo}_{\chi}(G)/ \La^{\fo_{\chi}}(\Ga_{\chi}')) = d_{\chi}$
  as claimed.

  Now let $\breve\La^{\fo}(G)$ be a second maximal $R$-order containing $\La^{\fo}(G)$.
  Then likewise $\breve\La^{\fo}(G)$ decomposes into $\bigoplus_{\chi \in \irr(G)/ \sim_K} \breve\La^{\fo}_{\chi}(G)$,
  where each $\breve\La^{\fo}_{\chi}(G)$ is a maximal $R$-order in $A_{\chi}$. We put
  $$\breve d_{\chi} := \mc Q^{K_{\chi}}(\Ga_{\chi}') \cap D(\breve \La^{\fo}_{\chi}(G)/ \La^{\fo_{\chi}}(\Ga_{\chi}')).$$
  Fix a height $1$ prime ideal $\fp \in P_{\chi}$.
  By \cite[Proposition 3.5]{AG_MO} there is a unit $a_{\chi} \in A_{\chi}$
  such that $\breve\La^{\fo}_{\chi}(G)_{\fp} = a_{\chi} \ti\La^{\fo}_{\chi}(G)_{\fp} a_{\chi}^{-1}$.
  Now let $x \in \mc Q^{K_{\chi}}(\Ga_{\chi}')$. Then
  \begin{eqnarray*}
        \tr_{A_{\chi} / \mc Q^{K_{\chi}}(\Ga_{\chi}')}(x \breve \La^{\fo}_{\chi}(G)_{\fp})
            & = & \tr_{A_{\chi} / \mc Q^{K_{\chi}}(\Ga_{\chi}')}(x a_{\chi} \ti \La^{\fo}_{\chi}(G)_{\fp} a_{\chi}^{-1}) \\
            & = & \tr_{A_{\chi} / \mc Q^{K_{\chi}}(\Ga_{\chi}')}(a_{\chi} x \ti \La^{\fo}_{\chi}(G)_{\fp} a_{\chi}^{-1}) \\
            & = & \tr_{A_{\chi} / \mc Q^{K_{\chi}}(\Ga_{\chi}')}(x \ti \La^{\fo}_{\chi}(G)_{\fp}),
  \end{eqnarray*}
  where the second equality holds, since $x$ is central in $A_{\chi}$.
  Hence
  $$\tr_{A_{\chi} / \mc Q^{K_{\chi}}(\Ga_{\chi}')}(x \breve \La^{\fo}_{\chi}(G)_{\fp}) \subseteq \La^{\fo_{\chi}}(\Ga_{\chi}')_{\fp}
  \iff \tr_{A_{\chi} / \mc Q^{K_{\chi}}(\Ga_{\chi}')}(x \ti \La^{\fo}_{\chi}(G)_{\fp}) \subseteq \La^{\fo_{\chi}}(\Ga_{\chi}')_{\fp},$$
  that is $(d_\chi)_{\fp} = (\breve d_{\chi})_{\fp}$
  for all $\fp \in P_{\chi}$. As both $d_{\chi}$ and $\breve d_{\chi}$ are reflexive, this implies $d_{\chi} = \breve d_{\chi}$
  as desired.
\end{proof}

Let $\pi_{\chi}$ be a prime element in $\fo_{\chi}$
and put $\fp_{\chi} := \pi_{\chi} \La^{\fo_{\chi}}(\Ga_{\chi}') \in P_{\chi}$.

\begin{lem} \label{lem:d_chi}
  There is an integer $r_{\chi} \leq 0$ such that $d_{\chi} = \fp_{\chi}^{r_{\chi}}$.
  If the Schur index $s_{\chi}$ is not divisible by $p$, then $r_{\chi} = 0$ and thus $d_{\chi} = \La^{\fo_{\chi}}(\Ga_{\chi}')$.
\end{lem}

\begin{proof}
  Let $\De_{\chi}$ be a maximal order in $D_{\chi}$. Then $(\De_{\chi})_{n_{\chi} \times n_{\chi}}$ is a maximal order in $A_{\chi}$
  by \cite[Theorem 8.7]{MO}. Let $\al \in (\De_{\chi})_{n_{\chi} \times n_{\chi}}$ be the matrix with a $1$ in the upper left corner
  and zeros everywhere else. Let $x \in d_{\chi}$. Then in particular
  \beq \label{eqn:tr-of-x}
    \tr_{A_{\chi} / \mc Q^{K_{\chi}}(\Ga_{\chi}')}(x \al) = x \cdot \tr_{A_{\chi} / \mc Q^{K_{\chi}}(\Ga_{\chi}')}(\al)
     = x \cdot s_{\chi} \in \La^{\fo_{\chi}}(\Ga_{\chi}').
  \eeq
  Here, the first equality follows from linearity of the reduced trace as $x$ belongs to $\mc Q^{K_{\chi}}(\Ga_{\chi}')$.
  For the second equality we compute
  \[
    \tr_{A_{\chi} / \mc Q^{K_{\chi}}(\Ga_{\chi}')}(\al) = \tr_{D_{\chi} / \mc Q^{K_{\chi}}(\Ga_{\chi}')}(1) = s_{\chi}.
  \]

  As $s_{\chi}$ is an integer, it follows from (\ref{eqn:tr-of-x}) that $x \in \La^{\fo_{\chi}}(\Ga_{\chi}')_{\fp}$
  for every $\fp \in P_{\chi}$, $\fp \not= \fp_{\chi}$.
  However, $d_{\chi}$ is reflexive by Lemma \ref{lem:independence} and thus we must have $d_{\chi} = \fp_{\chi}^{r_{\chi}}$ for some integer $r_{\chi}$.
  Since obviously $\La^{\fo_{\chi}}(\Ga_{\chi}') \subseteq d_{\chi}$, we have $r_{\chi} \leq 0$.
  Finally, if $p$ does not divide $s_{\chi}$, then by (\ref{eqn:tr-of-x}) we have
  $x \in \La^{\fo_{\chi}}(\Ga_{\chi}')_{\fp}$ for every $\fp \in P_{\chi}$ including $\fp_{\chi}$ and thus $r_{\chi} = 0$
  in this case.
\end{proof}

\begin{theo} \label{thm:central-conductor}
Let $\ti\La^{\fo}(G)$ be a maximal order containing $\La^{\fo}(G)$. Then $\mc F(\ti\La^{\fo}(G) / \La^{\fo}(G))$
does not depend on the choice of maximal order $\ti\La^{\fo}(G)$ and we have an equality
$$\mc F(\La^{\fo}(G)) = \mc F(\ti\La^{\fo}(G) / \La^{\fo}(G)) =
    \bigoplus_{\chi \in \irr(G)/ \sim_K} \frac{|H| w_{\chi}}{\chi(1)} \cdot \fD\me(\fo_{\chi} / \fo) d_{\chi}.$$
Here, $d_{\chi} = \fp_{\chi}^{r_{\chi}}$ for some integer $r_{\chi} \leq 0$.
In particular, we have an inclusion
\begin{equation} \label{eqn:conductor-inclusion}
  \frac{|H| w_{\chi}}{\chi(1)} \cdot \fD\me(\fo_{\chi} / \fo) \La^{\fo_{\chi}}(\Ga_{\chi}') \subseteq \mc F_{\chi}(\ti\La^{\fo}(G) / \La^{\fo}(G)),
\end{equation}
which becomes an equality after localization at every $\fp \in P_{\chi}$ different from $\fp_{\chi}$.
Moreover, the inclusion (\ref{eqn:conductor-inclusion}) is an equality in each of the following cases:
\ben[(i)]
    \item
    $G = H \times \Gamma$ is a direct product.
    \item
    The Schur index $s_{\chi}$ is not divisible by $p$.
    \item
    There is a maximal $\fo_{\chi}$-order $\De_{\chi}$ contained in $\ti\La^{\fo}_{\chi}(G)$
    such that $\La^{\fo_{\chi}}(\Ga_{\chi}') \otimes_{\fo_{\chi}} \De_{\chi} = \ti\La^{\fo}_{\chi}(G)$.
\een
\end{theo}

\begin{rem} \label{rem:condition-holds}
  Fix a $p$-adic Lie-group $G$ of dimension $1$. Then (ii) and (iii) hold for all $\chi$ whenever $K$ is sufficiently large.
  This follows from Theorem \ref{thm:splitting-field} and Lemma \ref{lem:first_technical}.
\end{rem}

\begin{rem}
  If $p$ does not divide the order of the commutator subgroup of $G$, then $\La^{\fo}(G)$ is a direct sum of
  matrix rings over commutative rings by \cite[Proposition 4.5]{JN_Fitting}. In particular, no skewfields occur in
  the Wedderburn decomposition of $\mc Q^K(G)$ and thus (ii) holds for every $\chi \in \irr(G)$.
\end{rem}

\begin{proof}[Proof of Theorem \ref{thm:central-conductor}]
Let us define
$$\de_{\chi}\me  :=  p^n \fD\me(\La^{\fo_{\chi}}(\Ga_{\chi}') / R) = w_{\chi} \cdot \fD\me(\fo_{\chi} / \fo) \La^{\fo_{\chi}}(\Ga_{\chi}'),$$
where the second equality follows from Lemma \ref{lem:different_welldef} and Remark \ref{rem:embedding}.

\begin{lem} \label{lem:different_eqn}
We have an equality $D_{\chi}(\ti\La^{\fo}(G)) = D(\ti \La^{\fo}_{\chi}(G)/ \La^{\fo_{\chi}}(\Ga_{\chi}')) \cdot \de_{\chi}\me$.
\end{lem}

\begin{proof}
This is a special case of Lemma \ref{lem:product_differents}.
\end{proof}

By Theorem \ref{thm:global_conductor} and the definition of the central conductor we obtain
$$\mc F_{\chi}(\ti\La^{\fo}(G) / \La^{\fo}(G)) =  \La^{\fo_{\chi}}(\Ga_{\chi}') \cap \left(\frac{|H|}{\chi(1)} \cdot D_{\chi}(\ti\La^{\fo}(G))\right)$$
for each $\chi \in \irr(G)$.
Hence it must be shown that
\beq \label{eqn:several_differents}
    \La^{\fo_{\chi}}(\Ga_{\chi}') \cap \left(\frac{|H|}{\chi(1)} \cdot D_{\chi}(\ti\La^{\fo}(G))\right) = \frac{|H|}{\chi(1)} \de_{\chi}\me d_{\chi}
\eeq
for each character $\chi$. We note that
$$\frac{|H|}{\chi(1)} \de_{\chi}\me d_{\chi} \subseteq\frac{|H|}{\chi(1)} D_{\chi}(\ti\La^{\fo}(G)) \subseteq\ti\La^{\fo}(G);$$
so each element of $\frac{|H|}{\chi(1)} \de_{\chi}\me d_{\chi}$ is integral over $R$, and thus lies in $\La^{\fo_{\chi}}(\Ga_{\chi}')$.
This gives one inclusion in (\ref{eqn:several_differents}).
Now let $y \in \La^{\fo_{\chi}}(\Ga_{\chi}')$. Then by Lemma \ref{lem:different_eqn} we have
$$y \in \frac{|H|}{\chi(1)} \cdot D_{\chi}(\ti\La^{\fo}(G))
    \iff y \de_{\chi} \subseteq\frac{|H|}{\chi(1)} D(\ti \La^{\fo}_{\chi}(G)/ \La^{\fo_{\chi}}(\Ga_{\chi}'))
    \iff y \de_{\chi} \subseteq\frac{|H|}{\chi(1)} d_{\chi}$$
giving the reverse inclusion in (\ref{eqn:several_differents}).
Lemma \ref{lem:independence} and Lemma \ref{lem:d_chi} imply the remaining assertions apart from the claim that
$r_{\chi} = 0$ in case (i) and (iii).

Let us first assume that $G = H \times \Ga$ is a direct product. Choose a maximal $\fo$-order $\fM(H)$ containing $\fo[H]$.
Let $\chi \in \irr(G)$. Then $\res^G_H \chi = \eta$ is an irreducible character of $H$ and
there is a character $\rho$ of type $W$ such that
$\chi \otimes \rho$ is trivial on $\Ga$. As $\chi \sim_K \chi \otimes \rho$, we may henceforth assume that $\rho=1$.
We then have $w_{\chi} = 1$, $K_{\chi} = K(\eta)$, $e_{\chi} = e(\eta)$ and $\ga_{\chi} = \ga e_{\chi} = \ga_{\chi}'$.
Then $\fM(H) \ve_{\chi}$ is a maximal $\fo_{\chi}$-order such that
$\ti\Lambda^{\fo}_{\chi}(G) := \La^{\fo_{\chi}}(\Ga) \otimes_{\fo_{\chi}} \fM(H) \ve_{\chi}$ is a maximal $R$-order
in $A_{\chi}$. It follows that (i) is a special case of (iii).

Now assume that (iii) holds and let $y \in (d_{\chi})_{\fp_{\chi}}$ be arbitrary.
We have to show that $y \in \La^{\fo_{\chi}}(\Ga_{\chi}')_{\fp_{\chi}}$.
As the reduced trace is $\mc Q^{K_{\chi}}(\Ga_{\chi}')$-linear,
we may alter $y$ by a unit in $\La^{\fo_{\chi}}(\Ga_{\chi}')_{\fp_{\chi}}$
so that we may assume
that $y = \pi_{\chi}^k$ for an appropriate integer $k$; in particular, we then have $y \in K_{\chi}$.
By hypothesis
there is a maximal $\fo_{\chi}$-order $\De_{\chi}$ contained in $\La^{\fo_{\chi}}(\Ga_{\chi}')$
such that $\mc Q^{K_{\chi}}(\Ga_{\chi}') \otimes_{\fo_{\chi}} \De_{\chi} = A_{\chi}$.
Hence an $\fo_{\chi}$-basis of $\De_{\chi}$ is also a $\mc Q^{K_{\chi}}(\Ga_{\chi}')$-basis of $A_{\chi}$ which we may use to compute
the reduced trace. Hence $y \in K_{\chi}$ has
$$\tr_{K_{\chi} \otimes_{\fo_{\chi}} \De_{\chi} / K_{\chi}} (y \De_{\chi}) \subseteq \La^{\fo_{\chi}}(\Ga_{\chi}')_{\fp_{\chi}} \cap K_{\chi} = \fo_{\chi}$$
and we have to show that $y \in \fo_{\chi}$; in other words we shall show that $K_{\chi} \cap \fD\me(\De_{\chi}/\fo_{\chi}) = \fo_{\chi}$.
Suppose that this is not true. Then $\fo_{\chi}$ is properly contained in $K_{\chi} \cap \fD\me(\De_{\chi}/\fo_{\chi})$.
Then $\fp_{\chi}\me \subseteq \fD\me(\De_{\chi}/\fo_{\chi})$
and thus $\fp_{\chi} \De_{\chi}$ contains $\fD(\De_{\chi}/\fo_{\chi})$. However, if $\fP_{\chi}$ denotes the radical of $\De_{\chi}$,
then $\fp_{\chi} \De_{\chi} = \fP_{\chi}^e$ for some positive integer $e$ and $\fD(\De_{\chi}/\fo_{\chi}) = \fP_{\chi}^{e-1}$
by \cite[Theorems 20.3 and 14.3]{MO}, a contradiction.
\end{proof}

\example Assume that $G = H \times \Ga$ is a direct product and thus
\beq \label{eqn:direct-product}
    \La^{\fo}(G) = \La^{\fo}(\Ga)[H] = \La^{\fo}(\Ga) \otimes_{\fo} \fo[H].
\eeq
Let $\chi \in \irr(G)$.
As we have seen in the proof of Theorem \ref{thm:central-conductor}
we then have $w_{\chi} = 1$, $K_{\chi} = K(\eta)$ and $\ga_{\chi}' = \ga e_{\chi}$. Hence
\[
    \mc F(\La^{\fo}(G))  =  \bigoplus_{\chi \in \irr(G) / \sim_K} \frac{|H|}{\chi(1)} \cdot \fD\me(\fo_{\chi} / \fo) \La^{\fo_{\chi}}(\Ga)
         =  \La^{\fo}(\Ga) \otimes_{\fo} \mc F(\fo[H])
\]
which can be shown more directly using (\ref{eqn:direct-product}) and Jacobinski's central conductor formula (\ref{eqn:central-conductor-J}).
So the main obstacle to derive Theorem \ref{thm:central-conductor} directly from Jacobinski's result is the fact that in general $G = H \rtimes \Ga$
is only semi-direct.\\

We finally determine the integers $r_{\chi}$ whenever $G$ is a pro-$p$-group. For simplicity we will also assume that $K=\qp$.
Then $K_{\chi} = \Q_{p,\chi}$ and $\fo_{\chi} =: \Z_{p,\chi}$ denotes the ring of integers in $\Q_{p,\chi}$.
Let $\fQ_{\eta}$ and $\fQ_{\chi}$ be the maximal ideals in $\zp[\eta]$ and $\Z_{p,\chi}$, respectively.
If $x$ is a rational number, we let $\lfloor x \rfloor \in \Z$ denote the largest integer such that $\lfloor x \rfloor \leq x$.

\begin{theo}
  Assume that $G$ is a pro-$p$-group and a $p$-adic Lie group of dimension $1$. Fix $\chi \in \irr(G)$
  and let $A_{\chi} \simeq (D_{\chi})_{n_{\chi} \times n_{\chi}}$ be the corresponding simple component of $\mc Q(G)$.
  Write $\fD(\zp[\eta]/\Z_{p,\chi}) = \fQ_{\eta}^{k_{\chi}}$ for some integer $k_{\chi} \geq 0$.
  Then $D_{\chi}$ is a cyclic skewfield with Schur index $s_{\chi} = [\qp(\eta):\Q_{p,\chi}]$ and
  $r_{\chi} = - \lfloor \frac{k_{\chi}}{s_{\chi}} \rfloor$.
\end{theo}

\begin{proof}
  By \cite[Theorem 1]{Irene_Lau} we know that $D_{\chi}$ is a cyclic skewfield with Schur index $s_{\chi} = [\qp(\eta):\Q_{p,\chi}]$.
  More precisely, the extension $\qp(\eta) / \Q_{p,\chi}$ is cyclic and totally ramified, as $G$ is pro-$p$ and thus
  $\qp(\eta)$ is a subfield of $\qp(\zeta_{p^m})$ for some positive integer $m$. Let $\si$ be a generator of the Galois group
  $\Gal(\qp(\eta) / \Q_{p,\chi})$. Then by \cite[Theorem 1]{Irene_Lau} the skewfield $D_{\chi}$ is given by
  \begin{equation} \label{eqn:skewfield-iso}
        D_{\chi} \simeq \bigoplus_{i=0}^{s_{\chi}-1} \left(\mc Q^{\qp(\eta)} (\Ga^{s_{\chi}}) \right) \ga^i,
  \end{equation}
  where $\ga^{s_{\chi}}$ is a topological generator of $\Ga^{s_{\chi}} \simeq \zp$ and for $x \in \mc Q^{\qp(\eta)} (\Ga^{s_{\chi}})$
  we have $\ga x = \si(x) \ga$ (note that this is not the same $\Ga$ as in $G = H \rtimes \Ga$;
  it is just a topological group isomorphic to $\zp$ which we again denote by $\Ga$). Then $D_{\chi}$ has center
  $\mc Q^{\Q_{p,\chi}}(\Ga^{s_{\chi}})$ which is isomorphic to $\mc Q^{\Q_{p,\chi}}(\Ga_{\chi}')$ under the isomorphism (\ref{eqn:skewfield-iso}).
  Now let $d = \sum_{i=0}^{s_{\chi}-1} x_i \ga^i \in D_{\chi}$ be arbitrary, $x_i \in \mc Q^{\qp(\eta)} (\Ga^{s_{\chi}})$, $0 \leq i < s_{\chi}$.
  We claim that
  \beq \label{eqn:trace-of-d}
    \tr_{D_{\chi} / \mc Q^{\Q_{p,\chi}}(\Ga^{s_{\chi}})}(d) = \Tr_{\mc Q^{\qp(\eta)} (\Ga^{s_{\chi}}) / \mc Q^{\Q_{p,\chi}} (\Ga^{s_{\chi}})}(x_0).
  \eeq
  For this one only has to check that we have an isomorphism
  \begin{eqnarray*}
    \qp(\eta) \otimes_{\Q_{p,\chi}} D_{\chi} & \simeq & (\mc Q^{\qp(\eta)}(\Ga))_{s_{\chi} \times s_{\chi}} \\
    x_0 & \mapsto & \left( \barr{cccc} x_0 & & & \\ & \si(x_0) & & \\ & & \ddots & \\ & & & \si^{s_{\chi}-1}(x_0) \earr \right) \\
    \ga & \mapsto & \left( \barr{cccc} 0 & & & \ga \\ \ga & \ddots & & 0 \\ & \ddots & 0 & \vdots \\ & & \ga & 0 \earr \right)
  \end{eqnarray*}
  From this one can easily compute the reduced trace in (\ref{eqn:trace-of-d}).
  Now put
  $$\Si_{\chi} := \bigoplus_{i=0}^{s_{\chi}-1} \left(\La^{\zp[\eta]} (\Ga^{s_{\chi}}) \right) \ga^i \subseteq D_{\chi}.$$
  Then $\Si_{\chi}$ is a $\La^{\Z_{p,\chi}} (\Ga^{s_{\chi}})$-order in $D_{\chi}$. Moreover,
  $\fQ_{\eta} \Si_{\chi}$ is a two-sided ideal of $\Si_{\chi}$, since $\ga \fQ_{\eta} = \si(\fQ_{\eta}) \ga = \fQ_{\eta} \ga$.
  We consider the $\La^{\Z_{p,\chi}} (\Ga^{s_{\chi}})_{\fp_{\chi}}$-order $\Si_{\fp_{\chi}} := (\Si_{\chi})_{\fp_{\chi}}$ in $D_{\chi}$.
  Then $\fQ_{\eta} \Si_{\fp_{\chi}}$ is again two-sided and we have
  $$(\fQ_{\eta} \Si_{\fp_{\chi}})^{s_{\chi}} = \fQ_{\chi} \Si_{\fp_{\chi}} = \fp_{\chi} \Si_{\fp_{\chi}} \subseteq \rad(\Si_{\fp_{\chi}}).$$
  Hence also $\fQ_{\eta} \Si_{\fp_{\chi}} \subseteq \rad(\Si_{\fp_{\chi}})$ by \cite[\S 6, Exercise 3]{MO}.
  We now observe that $\ga \mapsto 1+T$ induces an isomorphism
  $$\Si_{\chi} / \fQ_{\eta} \Si_{\chi} \simeq \F_p[[T]].$$
  Moreover, this isomorphism induces
  $$\Si_{\fp_{\chi}} / \fQ_{\eta} \Si_{\fp_{\chi}} \simeq \F_p((T))$$
  and thus the inclusion $\fQ_{\eta} \Si_{\fp_{\chi}} \subseteq \rad(\Si_{\fp_{\chi}})$
  is an equality.
  Moreover, the order $\Si_{\fp_{\chi}}$ is quasi-local.
  As $\fQ_{\eta}$ is a principal ideal, there is a short exact sequence
  \[
    0 \rightarrow \Si_{\fp_{\chi}} \rightarrow \Si_{\fp_{\chi}} \rightarrow \F_p((T)) \rightarrow 0
  \]
  and thus the projective dimension of $\F_p((T))$ considered as left $\Si_{\fp_{\chi}}$-module equals $1$.
  Now \cite[Corollary 1.3]{Ramras_MO} gives that also
  $\gldim (\Si_{\fp_{\chi}}) = 1$. It then follows from \cite[Theorem 2.3]{AG_MO} that
  $\Si_{\fp_{\chi}}$ is a maximal $\La^{\Z_{p,\chi}} (\Ga^{s_{\chi}})_{\fp_{\chi}}$-order. By (\ref{eqn:trace-of-d})
  we find that
  \begin{eqnarray*}
    \tr_{D_{\chi} / \mc Q^{\Q_{p,\chi}}(\Ga^{s_{\chi}})}(\Si_{\fp_{\chi}}) & = &
        \Tr_{\mc Q^{\qp(\eta)} (\Ga^{s_{\chi}}) / \mc Q^{\Q_{p,\chi}} (\Ga^{s_{\chi}})}(\La^{\zp[\eta]} (\Ga^{s_{\chi}})_{\fp_{\chi}}) \\
    & = & \Tr_{\qp(\eta) / \Q_{p,\chi}}(\zp[\eta]) \La^{\Z_{p,\chi}} (\Ga^{s_{\chi}})_{\fp_{\chi}}\\
    & = & \fQ_{\chi}^{r_{\chi}'} \La^{\Z_{p,\chi}} (\Ga^{s_{\chi}})_{\fp_{\chi}},
  \end{eqnarray*}
  where $r_{\chi}' \in \Z$ is maximal such that
  $$\zp[\eta] \subseteq \fQ_{\chi}^{r_{\chi}'} \fD(\zp[\eta]/\Z_{p,\chi})\me = \fQ_{\eta}^{s_{\chi} r_{\chi}' - k_{\chi}} $$
  and thus $r_{\chi}' = \lfloor \frac{k_{\chi}}{s_{\chi}} \rfloor$. Now let $x \in \mc Q^{\Q_{p,\chi}}(\Ga^{s_{\chi}})$.
  Then $x$ belongs to $(d_{\chi})_{\fp_{\chi}}$ if and only if
  $$\tr_{D_{\chi} / \mc Q^{\Q_{p,\chi}}(\Ga^{s_{\chi}})}(x \cdot \Si_{\fp_{\chi}}) = x \cdot \tr_{D_{\chi} / \mc Q^{\Q_{p,\chi}}(\Ga^{s_{\chi}})}(\Si_{\fp_{\chi}})
    \subseteq \La^{\Z_{p,\chi}} (\Ga^{s_{\chi}})_{\fp_{\chi}}$$
  and thus if and only if $x \in \fp_{\chi}^{- r_{\chi}'}$.
  Therefore we have $r_{\chi} = - r_{\chi}' = - \lfloor \frac{k_{\chi}}{s_{\chi}} \rfloor$ as desired.
\end{proof}

\example We continue the example in \cite[p. 1233]{Irene_Lau}. For this let $p=3$ and let $H$ be the cyclic group of order $9$.
Choose a generator $h$ of $H$. We put $G = H \rtimes \Ga$, where the action of $\Ga$ on $H$ is determined by $\ga h \ga\me = h^4$.
Let $\eta$ be the irreducible character of $H$ with $\eta(h) = \zeta_9$. Then $St(\eta) = H \times \langle \ga^3 \rangle$
and $\chi = \ind^G_{St(\eta)} \chi'$ with $\chi'(h) = \eta(h)$ and $\chi'(\ga^3) = 1$ is an irreducible character of $G$ with open kernel.
We have $\chi(1) = w_{\chi} = 3$ and $\qp(\eta) = \qp(\zeta_9)$, $\Q_{p,\chi} = \qp(\zeta_3)$.
We see that $s_{\chi} = [\qp(\zeta_9):\qp(\zeta_3)] = 3$.
Moreover, the different $\fD(\zp[\zeta_9]/\zp[\zeta_3])$ is easily computed and we find that $k_{\chi} = 3$.
Thus $r_{\chi} = - \lfloor \frac{k_{\chi}}{s_{\chi}} \rfloor = -1$ is non-trivial and (\ref{eqn:conductor-inclusion})
is a proper inclusion in this case.

\section{Consequences of the central conductor formula}

In this section we derive several consequences of our main Theorem \ref{thm:central-conductor}.


\subsection{Extensions of lattices}
\begin{cor} \label{cor:cond_ann_Ext}
Let $M$ be a $\La^{\fo}(G)$-lattice and let $N$ be a finitely generated $\La^{\fo}(G)$-module. Then
$$\left(\bigoplus_{\chi \in \irr(G)/ \sim_K} \frac{|H| w_{\chi}}{\chi(1)} \cdot \fD\me(\fo_{\chi} / \fo) d_{\chi} \right) \cdot \Ext^i_{\La^{\fo}(G)}(M,N)$$
is finite for all integers $i \geq 1$.
\end{cor}
\begin{proof}
An $R$-module $M$ is finite if and only if it is pseudo-null, i.e.~$M_{\fp} = 0$ for every height one prime ideal $\fp$ of $R$.
We do induction on the integer $i$. As $R_{\fp}$ is a Dedekind domain for every height one prime ideal $\fp$, the case $i=1$ is now immediate from Theorem \ref{thm:central-conductor} and \cite[Theorem 29.4]{CR-I}. For $k$ sufficiently large, there is an exact sequence
$$M' \into \La^{\fo}(G)^k \onto M.$$
Then $M'$ also is a $\La^{\fo}(G)$-lattice.
Applying $\Hom_{\La^{\fo}(G)}(-, N)$ to the above exact sequence gives isomorphisms
$$\Ext^j_{\La^{\fo}(G)}(M',N) \simeq \Ext^{j+1}_{\La^{\fo}(G)}(M,N)$$
for all integers $j \geq 1$. The case $j = i-1$ gives the induction step.
\end{proof}

\example The stronger statement that the central conductor annihilates $\Ext^i_{\La^{\fo}(G)}(M,N)$ is not true.
Assume that $G = \Ga$ and $K = \qp$. Then Corollary \ref{cor:cond_ann_Ext} simply asserts that $\Ext^i_{\La(\Ga)}(M,N)$ is finite
for every finitely generated torsionfree $\La(\Ga)$-module $M$ and every finitely generated $\La(\Ga)$-module $N$. Suppose that
$\Ext^i_{\La(\Ga)}(M,N) = 0$ for every finitely generated $\La(\Ga)$-module $N$; then $M$ is projective and thus free as $\La(\Ga)$-module.
However, there are many examples of torsionfree $\La(\Ga)$-modules which are not free (take for example the maximal ideal of $\La(\Ga)$).\\

For $\chi \in \irr(G)$ we put $\fa_{\chi}\me := (\fD\me(\fo_{\chi} / \fo) \cap K) \cdot R \supseteq R$.
\begin{cor} \label{cor:cond_cap_R}
We have
$$ \bigcap_{\chi \in \irr(G) / \sim_K} \frac{|H| w_{\chi}}{\chi(1)} \cdot \fa_{\chi}\me \subseteq R \cap \mc F(\La^{\fo}(G)).$$
\end{cor}

In fact, if $r_{\chi} = 0$ for all $\chi \in \irr(G)$, then a proof similar to that of \cite[Theorem 27.13(ii)]{CR-I} shows that
the inclusion in Corollary \ref{cor:cond_cap_R} is an equality.

For every $\La^{\fo}(G)$-module $M$ let $\Upsilon(M) := \left\{ e_{\chi}\mid e_{\chi} \cdot \mc Q^K(G) \otimes_{\La^{\fo}(G)} M = 0 \right\}$.
\begin{cor}
Let $M$ and $N$ be $\La^{\fo}(G)$-lattices and let $i\geq 1$ be an integer. Then
$$\left(\bigcap_{e_{\chi} \not\in \Upsilon(M)} \frac{|H| w_{\chi}}{\chi(1)} \cdot \fa_{\chi}\me \right) \Ext^i_{\La^{\fo}(G)}(M,N)$$
is finite. In particular,
$\left(\bigcap_{\chi \in \irr(G)/ \sim_K} \frac{|H| w_{\chi}}{\chi(1)} \cdot \fa_{\chi}\me \right)\cdot \Ext^i_{\La^{\fo}(G)}(M,N)$
is finite for every $\La^{\fo}(G)$-lattices $M$ and $N$ and every integer $i \geq 1$.
\end{cor}
\begin{proof}
The last assertion is an immediate consequence of Corollary \ref{cor:cond_ann_Ext} and Corollary \ref{cor:cond_cap_R}. The first assertion is also easy and is shown exactly
in the same way as \cite[Theorem 29.9]{CR-I}.
\end{proof}

\begin{cor}
Let $M$ and $N$ be $\La^{\fo}(G)$-lattices and assume that $\mc Q^K(G) \otimes_{\La^{\fo}(G)} M$ is absolutely simple.
Then there is a unique idempotent $e_{\chi} \not\in \Upsilon(M)$ and
for every integer $i\geq 1$ we have that
$\frac{|H| w_{\chi}}{\chi(1)} \cdot \Ext^i_{\La^{\fo}(G)}(M,N)$ is finite.
\end{cor}

\example Assume that $p$ divides $|H|$ and that $K=\qp$. We consider $\La(\Ga)$ as $\La(G)$-module in the natural way;
then $\mc Q(G) \otimes_{\La(G)} \La(\Ga)$ is absolutely simple
and $e_{\chi_0}$ is the unique idempotent not contained in $\Upsilon(\La(\Ga))$, where $\chi_0$ denotes the trivial character.
Moreover, we have $\chi_0(1) = 1$ and $w_{\chi_0} = 1$.
Let $\De(G,\Ga)$ be the kernel of the augmentation map $\aug: \La(G) \onto \La(\Ga)$ which sends each $h \in H$ to $1$.
We then have an exact sequence
$$\Hom_{\La(G)}(\La(\Ga), \De(G,\Ga)) \into \Hom_{\La(G)}(\La(\Ga), \La(G)) \stackrel{\aug_{\ast}}{\lto} \Hom_{\La(G)}(\La(\Ga), \La(\Ga))$$
    $$\onto \Ext^1_{\La(G)}(\La(\Ga), \De(G,\Ga)).$$
We have a canonical isomorphism $\Hom_{\La(G)}(\La(\Ga), \La(\Ga)) \simeq \La(\Ga)$, $f \mapsto f(1)$ and the image of $\aug_{\ast}$ in $\La(\Ga)$
under this identification is given by
$\aug((\sum_{h\in H} h) \La(G)) = |H| \La(\Ga)$. We therefore have a canonical isomorphism
$$\Ext^1_{\La(G)}(\La(\Ga), \De(G,\Ga)) \simeq \La(\Ga) / |H| \La(\Ga);$$
in particular, this module is not finite. However, $|H| \Ext^1_{\La(G)}(\La(\Ga), \De(G,\Ga)) = 0$ is finite.

%
%

\subsection{Non-commutative Fitting invariants}
For the following we refer the reader to \cite{ich-Fitting}.
Let $A$ be a separable $L$-algebra and $\La$ be an $R$-order in $A$, finitely generated as $R$-module,
where $R$ is an integrally closed complete commutative noetherian local domain with field of quotients $L$.
Let $N$ and $M$ be two $\zeta(\La)$-submodules of
    an $R$-torsionfree $\zeta(\La)$-module.
    Then $N$ and $M$ are called {\it $\nr(\La)$-equivalent} if
    there exists an integer $n$ and a matrix $U \in \Gl_n(\La)$
    such that $N = \nr(U) \cdot M$, where $\nr: A \to \zeta(A)$ denotes
    the reduced norm map which extends to matrix rings over $A$ in the obvious way.
    We denote the corresponding equivalence class by $[N]_{\nr(\La)}$.
    We say that $N$ is
    $\nr(\La)$-contained in $M$ (and write $[N]_{\nr(\La)} \subseteq [M]_{\nr(\La)}$)
    if for all $N' \in [N]_{\nr(\La)}$ there exists $M' \in [M]_{\nr(\La)}$
    such that $N' \subseteq M'$. 
    We will say that $x$ is contained in $[N]_{\nr(\La)}$ (and write $x \in [N]_{\nr(\La)}$) if there is $N_0 \in [N]_{\nr(\La)}$ such that $x \in N_0$.
    Now let $M$ be a finitely presented (left) $\La$-module and let
    \beq \label{finite_representation}
        \La^a \stackrel{h}{\lto} \La^b \onto M
    \eeq
    be a finite presentation of $M$.
    We identify the homomorphism $h$ with the corresponding matrix in $M_{a \times b}(\La)$ and define
    $S(h) = S_b(h)$ to be the set of all $b \times b$ submatrices of $h$ if $a \geq b$.
    The Fitting invariant of $h$ over $\La$ is defined to be
    $$\Fitt_{\La}(h) = \left\{ \barr{lll} [0]_{\nr(\La)} & \mbox{ if } & a<b \\
                        \left[\langle \nr(H) \mid H \in S(h)\rangle_{\zeta(\La)}\right]_{\nr(\La)} & \mbox{ if } & a \geq b. \earr \right.$$
    We call $\Fitt_{\La}(h)$ a Fitting invariant of $M$ over $\La$. One defines $\Fitt_{\La}^{\max}(M)$ to be the unique
    Fitting invariant of $M$ over $\La$ which is maximal among all Fitting invariants of $M$ with respect to the partial
    order ``$\subseteq$''.

    We now specialize to the situation in this article, where $\La$ is $\La^{\fo}(G)$.
    Then Theorem \ref{thm:central-conductor} and \cite[Corollary 6.2]{JN_Fitting} imply the following result.
    \begin{cor}
    Let $M$ be a finitely presented $\La^{\fo}(G)$-module. Then
    $$\left(\bigoplus_{\chi \in \irr(G)/ \sim_K} \frac{|H| w_{\chi}}{\chi(1)} \cdot \fD\me(\fo_{\chi} / \fo) d_{\chi}\right) \cdot \Fitt_{\La^{\fo}(G)}^{\max}(M)
    \subseteq\Ann_{\La^{\fo}(G)}(M).$$
    \end{cor}

    Together with  Corollary \ref{cor:cond_cap_R} and \cite[Lemma 3.4]{ich-Fitting} this yields the following corollary.
    \begin{cor}
    Let $M$ be a finitely presented $\La^{\fo}(G)$-module. Then
    $$\left(\bigcap_{e_{\chi} \not\in \Upsilon(M)} \frac{|H| w_{\chi}}{\chi(1)} \cdot \fa_{\chi}\me \right) \cdot \Fitt_{\La^{\fo}(G)}^{\max}(M)
    \subseteq\Ann_{\La^{\fo}(G)}(M).$$
    \end{cor}

    \begin{rem}
        Note that in fact $\mc H(\La^{\fo}(G)) \cdot \Fitt_{\La^{\fo}(G)}^{\max}(M) \subseteq\Ann_{\La^{\fo}(G)}(M)$, where $\mc H(\La^{\fo}(G))$
        is a certain ideal of $\zeta(\La^{\fo}(G))$ which always contains the central conductor. In general, however, this containment
        is not an equality. Though the ideal $\mc H(\La^{\fo}(G))$ is hard to determine in general, considerable progress is made in \cite{JN_Fitting};
        in particular, by \cite[Proposition 4.5]{JN_Fitting} one knows that $\mc H(\La^{\fo}(G))$ equals $\zeta(\La^{\fo}(G))$
        (to wit: is best possible) if and only if $p$
        does not divide the order of the commutator subgroup of $G$ (which is finite).
    \end{rem}

\noindent Andreas Nickel~~ anickel3@math.uni-bielefeld.de\\
Universit\"{a}t Bielefeld,
    Fakult\"{a}t f\"{u}r Mathematik,
    Postfach 100131,
    33501 Bielefeld,
    Germany

\end{document}